\documentclass[11pt]{amsart}

\usepackage[a4paper,hmargin=3.5cm,vmargin=4cm]{geometry}
\usepackage{amsfonts,amssymb,amscd,amstext}

\usepackage{fancyhdr}
\input xy
\xyoption{all}

\usepackage{times}

\usepackage{enumerate}
\usepackage{titlesec}
\usepackage{mathrsfs}

\titleformat{\section}
{\filcenter\bfseries\large} {\thesection{.}}{0.2cm}{}
\titleformat{\subsection}[runin]
{\bfseries} {\thesubsection{.}}{0.15cm}{}[.]
\titleformat{\subsubsection}[runin]
{\em}{\thesubsubsection{.}}{0.15cm}{}[.]

\usepackage[up,bf]{caption}

\newtheorem{theorem}{Theorem}[section]

\newtheorem{lemma}[theorem]{Lemma}
\newtheorem{corollary}[theorem]{Corollary}

\theoremstyle{definition}
\newtheorem{definition}[theorem]{Definition}
\newtheorem{remark}[theorem]{Remark}

\newtheorem{example}[theorem]{Example}

\numberwithin{equation}{section}
\numberwithin{figure}{section}

\newcommand\Acal{\mathcal{A}}
\newcommand\Bcal{\mathcal{B}}

\newcommand\Pcal{\mathcal{P}}

\newcommand\Ascr{\mathscr{A}}

\newcommand\Cscr{\mathscr{C}}

\newcommand\Fscr{\mathscr{F}}

\newcommand\Oscr{\mathscr{O}}

\renewcommand\c{\mathbb{C}}

\renewcommand\r{\mathbb{R}}

\newcommand\z{\mathbb{Z}}

\newcommand\B{\mathbb{B}}
\newcommand\C{\mathbb{C}}
\newcommand\D{\mathbb{D}}

\newcommand\N{\mathbb{N}}
\newcommand\R{\mathbb{R}}
\renewcommand\S{\mathbb{S}}

\newcommand\Z{\mathbb{Z}}

\newcommand\igot{\mathfrak{i}}

\renewcommand\igot{\mathfrak{i}}

\renewcommand\imath{\igot}

\newcommand\Agot{\mathfrak{A}}
\newcommand\Igot{\mathfrak{I}}

\newcommand\Mgot{\mathfrak{M}}
\newcommand\Ngot{\mathfrak{N}}

\newcommand\Ogot{\mathfrak{O}}

\renewcommand\span{\mathrm{span}}

\newcommand\wt{\widetilde}

\newcommand\di{\partial}
\newcommand\dibar{\overline\partial}
\newcommand\hra{\hookrightarrow}
\newcommand\longhookrightarrow{\ensuremath{\lhook\joinrel\relbar\joinrel\rightarrow}}
\newcommand\lra{\longrightarrow}

\newcommand\Flux{\mathrm{Flux}}

\newcommand\GM{\mathfrak{GM}}

\newcommand\Co{\mathrm{Co}}

\begin{document}

\fancyhead[LO]{The parametric h-principle for minimal surfaces in $\R^n$ and null curves in $\C^n$}
\fancyhead[RE]{Franc Forstneri\v c and Finnur L\'arusson}
\fancyhead[RO,LE]{\thepage}

\thispagestyle{empty}

\vspace*{7mm}
\begin{center}
{\bf \LARGE The parametric h-principle \\ \vspace{2mm} for minimal surfaces in $\R^n$ and null curves in $\C^n$}
\vspace*{5mm}

{\large\bf Franc Forstneri\v c and Finnur L\'arusson}
\end{center}

\vspace*{7mm}

\begin{quote}
{\small
\noindent {\bf Abstract}\hspace*{0.1cm}
Let $M$ be an open Riemann surface.  It was proved by Alarc\'on and Forstneri\v c \cite{AlarconForstnericCrelle}
that every conformal minimal immersion $M\to\R^3$ is isotopic to the real part of a holomorphic null curve $M\to\C^3$.
In this paper, we prove the following much stronger result in this direction: for any $n\ge 3$, the inclusion $\iota:\Re \Ngot_*(M,\C^n) \hra \Mgot_*(M,\R^n)$ of the space of real parts of nonflat null holomorphic immersions $M\to\C^n$ into the space of nonflat conformal minimal immersions $M\to \R^n$ satisfies the parametric h-principle with approximation (see Theorem \ref{th:PHP1}).  In particular, $\iota$ is a weak homotopy equivalence (see Theorem \ref{th:WHE1}).  We prove analogous results for several other related maps (see Theorems \ref{th:WHE2} and \ref{th:WHE-A} and Corollary \ref{cor:NC}),
%
%
and we describe the rough shape of the space of all holomorphic immersions $M\to\C^n$
(Theorem \ref{th:immersions}).  
%
%
For an open Riemann surface $M$ of finite topological type, we obtain optimal results by
showing that $\iota$ and several related maps are inclusions of strong deformation retracts; in particular, they are homotopy equivalences (see Corollary \ref{cor:strong} and Remark \ref{rem:general}).

\vspace*{0.1cm}
\noindent{\bf Keywords}\hspace*{0.1cm} Riemann surface, minimal surface, null curve, h-principle, Oka manifold, absolute neighborhood retract. 

\vspace*{0.1cm}

\noindent{\bf MSC (2010)}\hspace*{0.1cm} Primary 53A10.  Secondary 30F99, 32E30, 32H02, 54C55, 55M15.

\vspace*{0.1cm}

\noindent{\bf Date}\hspace*{0.1cm} 3 February 2016; this version 31 October 2016}
\end{quote}

%
%
%
%
\section{Introduction}\label{sec:intro}

This paper brings together four diverse topics from differential geometry, holomorphic geometry, and topology; namely the theory of minimal surfaces, Oka theory, convex integration theory, and the theory of absolute neighborhood retracts.  
Our goal is to determine the rough shape of several spaces of maps of great geometric interest.  
It turns out that they all have the same rough shape.  

We start by recalling some basic definitions and establishing notation.  Let $M$ be an open Riemann surface, 
and let $n\ge 3$ be an integer. 
It is a well known elementary observation (see e.g.\ Osserman \cite{Ossermanbook}) that a smooth immersion 
$u=(u_1,\ldots,u_n)\colon M\to \R^n$ is {\em conformal} (i.e., angle preserving) 
if and only if its $(1,0)$-differential $\di u=(\di u_1,\ldots, \di u_n)$ 
satisfies the following {\em nullity condition}:
\begin{equation}\label{eq:nullity}
	(\di u_1)^2 + (\di u_2)^2 + \cdots + (\di u_n)^2 = 0.
\end{equation}
Furthermore, a conformal immersion $u\colon M\to \R^n$ is {\em minimal}
(i.e., it parameterizes a minimal surface in $\R^n$) if and only if it is harmonic, $\triangle u =0$, and 
this holds if and only if $\di u$ is a holomorphic $(1,0)$-form. 
Such an immersion $u$ is said to be {\em nonflat} if, for each connected component $M'$ of $M$, 
the image $u(M')\subset\R^n$ is not contained in any affine $2$-plane. 
We denote by $\Mgot(M,\R^n)$ the space of all conformal minimal immersions $M\to\R^n$
with the compact-open topology, and by 
\[
	\Mgot_*(M,\R^n) \subset \Mgot(M,\R^n)
\] 
the subspace of $\Mgot(M,\R^n)$ consisting of all nonflat conformal minimal immersions. 

A holomorphic immersion $F\colon M\to \C^n$ $(n\ge 3)$ is a  {\em null curve}
if the differential $dF=\di F=(dF_1,\ldots, dF_1)$ satisfies the nullity condition \eqref{eq:nullity}.
Such an immersion $F$ is said to be {\em nonflat} if for each connected component $M'$ of $M$,
the image $F(M')$ is not contained in an affine complex line.
Since $dF=2\di (\Re F)$,  the real part $\Re F$ of a (nonflat) 
null curve is a (nonflat) conformal minimal immersion $M\to \R^n$;
the converse holds if $M$ is simply connected. Hence, we have inclusions 
\[ 
	\Re \Ngot(M,\C^n) \hra \Mgot(M,\R^n), \quad \Re \Ngot_*(M,\C^n) \hra \Mgot_*(M,\R^n),
\]
where $\Ngot_*(M,\C^n) \subset \Ngot(M,\C^n)$ is the space of all (nonflat)
null holomorphic  immersions $M\to\C^n$ with the compact-open topology. Here,  
$\Re \Fscr$ stands for the set of real parts of maps in a space $\Fscr$. 

The following is our first main result. 

%
%
\begin{theorem}\label{th:WHE1}
Let $M$ be an open Riemann surface. For every $n\ge 3$, the inclusion 
\begin{equation}\label{eq:inclusion}
	\Re \Ngot_*(M,\C^n) \longhookrightarrow \Mgot_*(M,\R^n)
\end{equation}
of the space of real parts of nonflat null holomorphic immersions $M\to\C^n$ into the space 
of nonflat conformal minimal immersions $M\to \R^n$ is a weak homotopy equivalence.
\end{theorem}

This means that the inclusion induces a bijection of path components and an isomorphism of homotopy groups
\[
	\pi_k \left(\Re \Ngot_*(M,\C^n)\right) \stackrel{\cong} {\longrightarrow} \pi_k \left(\Mgot_*(M,\R^n)\right)
\]
for every $k\geq 1$ and every choice of base point.

Theorem \ref{th:WHE1} is an immediate consequence of Theorem \ref{th:PHP1} which 
shows that the inclusion \eqref{eq:inclusion}  enjoys the parametric h-principle with approximation.
Theorem \ref{th:PHP1} is an analogue of the {\em parametric Oka property with approximation}
for the inclusion $\Oscr(X,Y)\hookrightarrow \Cscr(X,Y)$, where $X$ is a Stein manifold and $Y$
is an Oka manifold (see \cite[Theorem 5.4.4]{Forstneric2011}). This relationship 
plays an important role in the proof of Theorem \ref{th:WHE1} which relies on techniques
of modern Oka theory, combined with Gromov's convex integration theory.

The basic case of Theorem \ref{th:WHE1}, with $P=\{p\}$ a singleton and $Q=\emptyset$ (see the notation in Theorem \ref{th:PHP1}), was proved
by Alarc\'on and Forstneri\v c in \cite[Theorem 1.1]{AlarconForstnericCrelle}. In this case,
the result says that every nonflat conformal minimal immersion $M\to\R^n$ is isotopic
through a family of nonflat conformal minimal immersions $M\to\R^n$ to the real part of a holomorphic 
null curve $M\to\C^n$. In fact, Theorem \ref{th:WHE1} gives an affirmative answer to the second question 
in \cite[Problem 8.1]{AlarconForstnericCrelle}.

In this paper, we shall systematically use the term {\em isotopy} instead of the more standard 
{\em regular homotopy} when speaking of smooth 1-parameter families of immersions. 

We do not know whether the inclusion $\Re \Ngot(M,\C^n) \hra \Mgot(M,\R^n)$
(i.e., with the flat conformal minimal immersions included) is also a weak homotopy 
equivalence. The main problem is that flat conformal minimal immersions (and flat null curves)
are critical points of the period map (see \eqref{eq:period}), and hence they may be singular points 
of the space $\Mgot(M,\R^n)$.  This phenomenon was already observed in the papers \cite{AlarconForstnericCrelle,AlarconForstneric2014IM}. 

Let us recall the classical {\em Weierstrass representation} of conformal minimal
immersions and null curves; see e.g.\ \cite{Ossermanbook}.
Choose a nowhere vanishing holomorphic $1$-form $\theta$ on $M$.
(Such a form exists by the Oka-Grauert principle, cf.\ \cite[Theorem 5.3.1]{Forstneric2011}.)
Given a conformal minimal immersion $u\colon M\to \R^n$, the map
\[
	f=2\di u/\theta =(f_1,\ldots, f_n)\colon M\to \C^n
\] 
is holomorphic since $u$ is harmonic, and the nullity condition
\eqref{eq:nullity} shows that it has range in the punctured  {\em null quadric}
$\Agot_*=\Agot\setminus\{0\}$, where 
%
%
\begin{equation}\label{eq:Agot}
	\Agot = \Agot^{n-1} = \{(z_1,\ldots,z_n) \in\c^n \colon z_1^2+z_2^2+\cdots + z_n^2=0\}.
\end{equation}
Clearly, the choice of $\theta$ is immaterial since $\Agot$ is a cone.
Conversely, a holomorphic map $f\colon M\to \Agot_*$ such that the $(1,0)$-form $f\theta$
has vanishing real periods (i.e., $\int_\gamma \Re(f\theta)=0$ for every closed curve $\gamma$ in $M$)
determines a conformal minimal immersion $u\colon M\to\R^n$ defined by 
\[
	u(x)= \int^x \Re(f\theta),\quad x\in M. 
\]
Similarly, if $f\theta$ has vanishing complex periods, then it integrates to a holomorphic null curve $F(x)=\int^x f\theta$.

%
%
\begin{theorem}\label{th:WHE2}
Let $M$ be an open Riemann surface and let $n\ge 3$.  The maps
\begin{equation}\label{eq:WHE2}
	\Mgot_*(M,\R^n)  \lra  \Oscr(M,\Agot_*), \qquad 
	\Ngot_*(M,\C^n)   \lra  \Oscr(M,\Agot_*),
\end{equation}
given by $u\mapsto 2\di u/\theta$ and $F\mapsto \di F/\theta$, respectively, are weak homotopy equivalences. 
\end{theorem}

We use the standard notation $\Oscr(M,Y)$ for the space of all holomorphic maps from $M$ to a complex manifold $Y$
with the compact-open topology.

Before commenting on the proof of Theorem \ref{th:WHE2}, let us indicate a corollary.

Recall  (cf.\ \cite[Example 4.4]{AlarconForstneric2014IM})
that the punctured null quadric $\Agot_*$ is elliptic in the sense of Gromov \cite{Gromov1989}, and hence
an {\em Oka manifold} (see \cite[Corollary 5.5.12]{Forstneric2011}). This is also seen by observing 
that $\Agot_*$ is a homogeneous space of the complex Lie group $\C^*\times O(n,\C)$, where 
\[
	O(n,\C)=\{A\in GL(n,\C): AA^t=I\}
\] 
is the orthogonal group over $\C$. Recall that every complex homogeneous manifold is an Oka manifold 
by Grauert's theorem \cite{Grauert1957}; see also  \cite[Proposition  5.5.1]{Forstneric2011}.
Hence, it follows from Grauert's Oka principle (see \cite{GrauertMA1958}) that the inclusion 
\[
	\Oscr(M,\Agot_*) \longhookrightarrow  \Cscr(M,\Agot_*)
\]
of the space of holomorphic maps $M\to \Agot_*$  into the space of continuous maps   
is a weak homotopy equivalence. (See also \cite[Corollary 5.4.8]{Forstneric2011}.) 
Since the composition of weak homotopy equivalences is again a weak homotopy equivalence, 
we have the following corollary to Theorem \ref{th:WHE2}, generalizing \cite[Corollary 8.3]{AlarconForstnericCrelle}.

%
%
\begin{corollary}\label{cor:NC}
Let $M$ be an open Riemann surface, and let $\Agot$ be the null quadric in $\C^n$
for some  $n\ge 3$. Then the maps 
\[
	\Mgot_*(M,\R^n)  \lra \Cscr(M,\Agot_*),\qquad \Ngot_*(M,\C^n)   \lra \Cscr(M,\Agot_*),
\]
defined as in Theorem \ref{th:WHE2}, are weak homotopy equivalences.
\end{corollary}

Let us return to Theorem \ref{th:WHE2}. Consider the following commuting diagram: 
\begin{equation}\label{eq:diagram}
\xymatrix{
	\Ngot_*(M,\C^n)  \ar[r]^\phi \ar[d]  &  \Oscr(M,\Agot_*)  \\ 
	\Re\Ngot_*(M,\C^n)   \ar@{^{(}->}[r]     &  \Mgot_*(M,\R^n)  \ar[u]_\psi
}
\end{equation}
The inclusion  at the bottom is a weak homotopy equivalence by
Theorem \ref{th:WHE1}. The left vertical map is the projection $F\mapsto \Re F$ of a null curve to its real part.  By continuity 
in the compact-open topology of the Hilbert transform that takes $u\in\Re\Ngot_*(M,\C^n)$ to its harmonic conjugate $v$ with $v(p)=0$, where $p\in M$ is any chosen base point, the left vertical map is a homotopy equivalence.  The map $\phi$ is given by $F\mapsto \di F/\theta$, and the map $\psi$ is given by $u\mapsto 2\di u/\theta$. Hence, if one of the maps $\phi$, $\psi$ is a weak homotopy equivalence, then so is the other one. To prove Theorem \ref{th:WHE2}, it thus suffices to consider only one of them.
 
The fact that the map $\phi\colon \Ngot_*(M,\C^n)   \to  \Oscr(M,\Agot_*)$
is a weak homotopy equivalence is a special case of Theorem \ref{th:WHE-A}.
The latter result establishes the weak homotopy equivalence principle 
for the space $\Igot_{A,*}(M,\C^n)$ of nondegenerate holomorphic immersions $M\to\C^n$ 
directed by any conical complex subvariety $A\subset \C^n$ such that
$A_*=A\setminus \{0\}$ is an Oka manifold; 
the null quadric $\Agot$ defined by \eqref{eq:Agot} is a special case.
The technical result behind it is Theorem \ref{th:PHP2} which establishes the 
parametric h-principle with approximation in this context, generalizing the basic h-principle 
\cite[Theorem 2.6]{AlarconForstneric2014IM}.

These results also hold in the special case when the cone $A$ equals $\C^n$.
This deserves particular attention, so let us explain it in some details. 
For any integer $n\ge 1$, we denote by
\[
	\Igot(M,\C^{n}) 
\]
the subset of $\Oscr(M,\C^n)$ consisting of all holomorphic immersions $M\to\C^{n}$.
(Note that every immersion $M\to \C^{n}$ is nondegenerate in the sense of 
Definition \ref{def:nondegenerate}.) Fix a nowhere vanishing holomorphic $1$-form $\theta$ on $M$.
For every $F\in \Igot(M,\C^{n})$, the map
\[	
	\vartheta_F = dF/\theta\colon M\to \C^n_* =\C^n\setminus\{0\}
\]
is holomorphic and its range avoids the origin. This defines a continuous map
\[ \vartheta:\Igot(M,\C^{n}) \to \Oscr(M,\C^{n}_*). \]


The natural inclusion $\iota\colon\Oscr(M,\C^{n}_*)\hookrightarrow \Cscr(M,\C^{n}_*)$ 
is a weak homotopy equivalence by the Oka-Grauert principle \cite[Theorem 5.3.2]{Forstneric2011};  
if $M$ has  finite topological type, then it is even a homotopy equivalence \cite{Larusson2015}.
Further, the projection $\C^{n}_*\to \S^{2n-1}$, $z\mapsto z/|z|$, of $\C^{n}_*$ onto the unit sphere 
$\S^{2n-1}\subset \C^{n}$ induces a homotopy equivalence 
$\tau\colon \Cscr(M,\C^{n}_*)\to \Cscr(M,\S^{2n-1})$. In summary, we have maps
\[ 
	\Igot(M,\C^{n}) \stackrel{\vartheta}{\longrightarrow} \Oscr(M,\C^{n}_*) 
	\stackrel{\iota}{\longhookrightarrow} \Cscr(M,\C^{n}_*) \stackrel{\tau}{\longrightarrow} \Cscr(M,\S^{2n-1}),
\] 
where $\iota$ and $\tau$ are known to be (weak) homotopy equivalences.

The following Oka principle for holomorphic immersions from open Riemann surfaces to Euclidean spaces
is a special case of Theorem \ref{th:WHE-A}. 

%
%
\begin{theorem} \label{th:immersions}
For every open Riemann surface $M$, the maps  $\vartheta : \Igot(M,\C^{n}) \to  \Oscr(M,\C^{n}_*)$ and
\[
	\vartheta/|\vartheta|= \tau\circ\iota\circ \vartheta\colon \Igot(M,\C^{n})\longrightarrow \Cscr(M,\S^{2n-1})
\]
are weak homotopy equivalences; they are homotopy equivalences if the surface $M$ has  finite topological type.
\end{theorem}

The underlying technical result (the parametric Oka principle with approximation) is 
Theorem \ref{th:PHP2}. For the last statement concerning homotopy equivalences, 
see Corollary \ref{cor:strong} and Remark \ref{rem:general}.
From Theorem \ref{th:immersions} and standard homotopy theory we obtain the following corollary.
(Compare with Corollary \ref{cor:components}.)

\begin{corollary}\label{cor:immersions}
The path components of the space $\Igot(M,\C)$ are in bijective correspondence
with the elements of $H^1(M;\Z)=\Z^\ell$, and $\Igot(M,\C^{n})$ is $(2n-3)$-connected if $n>1$.
\end{corollary}

Holomorphic immersions from an open Riemann surface $M$ to $\C$ were treated by 
Gunning and Narasimhan \cite{GunningNarasimhan1967MA}.
They proved that every continuous map $M\to \C_*=\C\setminus\{0\}$ is homotopic to the derivative $dF/\theta$ 
of a holomorphic immersion $F\colon M\to\C$. (Note that such $F$ is a holomorphic function on $M$
without critical points.) However, to the best of our knowledge, the homotopy type of the space of 
holomorphic immersions $M\to\C^n$ has not been considered in the literature.

The basic (nonparametric) case of Theorem \ref{th:immersions} is a special case of the following 
basic h-principle for holomorphic immersions of Stein manifolds to complex Euclidean spaces, 
due to Eliashberg and Gromov \cite[Sec.\ 2.1.5]{Gromov1986book}:

{\em Suppose that the cotangent bundle $T^*X$ of a Stein manifold $X$  
is pointwise generated by $(1, 0)$-forms $\phi_1,\ldots, \phi_n$ for some $n>\dim X$.
Then the $n$-tuple $\phi=(\phi_1,\ldots, \phi_n)$ can be changed by a homotopy 
of such $n$-tuples generating $T^*X$  to the differential $df=(df_1, \ldots, df_n)$
of a holomorphic immersion $f=(f_1,\ldots,f_n)\colon X\to\C^n$.}

A detailed proof of this result and an extension to the $1$-parametric case was given by 
Kolari\v c \cite[Theorem 1.1]{Kolaric2011DGA}. He also provided a parametric version
in which the parameter space is a Stein manifold \cite[Theorem 1.3]{Kolaric2011DGA}. 
However, his result does not seem to give the full parametric h-principle 
(and, by extension, the weak homotopy equivalence principle) for holomorphic 
immersions $X\to \C^n$ when $n>\dim X$. The main problem is that the 
parametric h-principle concerns the existence of certain homotopies for
suitable pairs of compact parameter spaces $Q\subset P$, where the homotopy is fixed 
on $Q$. 

It is still an open problem whether the Eliashberg-Gromov theorem remains true for $n=\dim X$
(see \cite[Problem 8.11.3 and Theorem 8.12.4]{Forstneric2011}), except in 
the case $n=1$ when the problem was solved affirmatively by Gunning and Narasimhan
\cite{GunningNarasimhan1967MA}.

%
%

Although we have stated our results for an open Riemann surface, we wish to point out that the 
analogous results hold when $M$ is a compact bordered Riemann surface and we consider
null curves $M\to\C^n$ and conformal minimal immersions $M\to\R^n$ of class $\Cscr^r(M)$ 
for some $r\ge 1$. This will be clear from the proofs which proceed by induction over an exhausion
of the given open Riemann surface by compact smoothly bounded Runge domains.
We refer to \cite[Theorem 4.1]{AlarconForstnericCrelle} for the basic (nonparametric) case in this setting.

In the final section of the paper, we assume that the open Riemann surface $M$ has finite topological type.  This means that the fundamental group of $M$ is finitely generated or, equivalently, that $M$ can be obtained from a compact Riemann surface by removing a finite number of mutually disjoint points and closed discs.  Using the theory of absolute neighborhood retracts as in \cite{Larusson2015} and results, originating in \cite{Forstneric2007}, on certain spaces of maps being Banach manifolds, we are able to upgrade the inclusion in Theorem \ref{th:WHE1}, not only to a homotopy equivalence, but in fact to the inclusion of a strong deformation retract.  The same holds for the maps in Theorem \ref{th:WHE2} and Corollary \ref{cor:NC} with their sources and targets appropriately modified (see Corollary \ref{cor:strong} and Remark \ref{rem:general}).

Recall the theorem of Alarc\'on and Forstneri\v c \cite[Theorem 1.1]{AlarconForstnericCrelle} that every nonflat conformal minimal immersion $u:M\to\R^n$ is isotopic through such maps to the real part of a holomorphic null curve $M\to\C^n$.  By Corollary \ref{cor:strong}, when $M$ has finite topological type, this deformation can be carried out for all $u$ at once, in a way that depends continuously on $u$ and leaves $u$ unchanged if it is the real part of a holomorphic null curve to begin with.  The key to this result is the fact that the spaces $\Re\Ngot_*(M,\C^n)$ and $\Mgot_*(M,\R^n)$ are absolute neighborhood retracts when $M$ has finite topological type (see Theorem \ref{th:anr}).  The proof requires a parametric h-principle with approximation on suitable compact subsets of $M$.  It is for this purpose that we have included approximation in the parametric h-principles in Theorems \ref{th:PHP1}, \ref{th:PHP2} and \ref{th:PHP3}.

We conclude the introduction by indicating an application of Corollary \ref{cor:NC} to the question of identifying
path connected components of the mapping spaces under consideration. This question was raised in 
\cite[Section 8]{AlarconForstnericCrelle} where a partial answer was provided.

Recall that the first homology group $H_1(M;\Z)$ of an open Riemann surface 
is a free abelian group on countably many generators:
$H_1(M;\Z)\cong \Z^l$ with $l\in \Z_+\cup\{\infty\}$. (Here, $\Z_+=\{0,1,2,\ldots\}$.)
On the target side, it is easily seen that the
punctured null quadric $\Agot_*^{n-1}\subset \C^n$ is simply connected when $n\ge 4$,
while $\pi_1(\Agot_*^2)\cong \Z_2$ (see \cite[Equation (8.3)]{AlarconForstnericCrelle}). 
The latter fact is also seen by observing that $\Agot_*^2$ admits the unbranched 
two-sheeted holomorphic covering 
\begin{equation}\label{eq:spinorial}
	\pi\colon \C^2_*=\C^2\setminus\{0\}\to \Agot^2_*,\quad \pi(u,v)=\big(u^2-v^2,\imath(u^2+v^2),2uv\big)
\end{equation}
from the simply connected space $\C^2_*$.  (Here, $\imath = \sqrt{-1}$.)  The map $\pi$, or its variants, is called the {\em spinorial representation formula} for the null quadric in $\C^3$.
It follows that the path connected components of the space $\Cscr(M,\Agot_*^2)$ are in one-to-one correspondence with 
the elements of the free abelian group $(\Z_2)^l$ (see \cite[Proposition 8.4]{AlarconForstnericCrelle}),
and $\Cscr(M,\Agot_*^{n-1})$ is path connected if $n\ge 4$. 
Hence, Corollary \ref{cor:NC} implies the following result. (A partial result in this
direction is given by \cite[Proposition 8.4]{AlarconForstnericCrelle}.)

%
%
%
%
\begin{corollary} \label{cor:components}
Let $M$ be a connected open Riemann surface with $H_1(M;\Z)\cong \Z^l$ for some $l\in \Z_+\cup\{\infty\}$. 
Then the path connected components of each of the spaces $\Mgot_*(M,\R^3)$ and $\Ngot_*(M,\C^3)$
are in one-to-one correspondence with the elements of the abelian group $(\Z_2)^l$. 
If $n\ge 4$, then the spaces $\Mgot_*(M,\R^n)$ and $\Ngot_*(M,\C^n)$ are path connected.
\end{corollary}

We illustrate Corollary \ref{cor:components} by a few examples in dimension $n=3$.

Let $\pi\colon \C_*^2 \to\Agot_*$ be the universal covering map \eqref{eq:spinorial}.
Let $M$ be $\C_*=\C\setminus \{0\}$ or an annulus.  
Since $\Agot_*$ is an Oka manifold and $\pi_1(\Agot_*)=\Z_2$,
there are precisely two homotopy classes of holomorphic maps $f\colon M\to \Agot_*$.
Note that $f$ is nullhomotopic if and only if it factors through $\pi$
(by a continuous, or equivalently holomorphic, map). 
Assume now that $f=(f_1,f_2,f_3)$ is the derivative of a holomorphic null curve $F:M\to\C^3$
and $f_1\neq \imath f_2$. Consider the Weierstrass representation
(see Osserman \cite[Lemma 8.1, p.\ 63]{Ossermanbook}):
\[ 
	f_1=(1-g^2)\eta, \quad f_2=\imath (1+g^2)\eta, \quad f_3=2g\eta, 
\]
where $g$ is meromorphic and $\eta$ is holomorphic on $M$.  
Assume for simplicity that $g$ is holomorphic or, equivalently, that $\eta$ has no zeros.
Then $f$ factors through $\pi$ if and only if $\eta$ has a square root on $M$.
Indeed, if $\eta$ has a square root then $f=\pi(\sqrt{\eta},g\sqrt{\eta})$;
conversely, if $f=\pi(\sigma,\tau)$ for some holomorphic map $(\sigma,\tau):M\to\C_*^2$, 
then $\sigma^2=\eta$.

\begin{example}
1. {\em A flat null curve:}
Let $M=\C_*$, and let $f \colon\C_*\to \Agot_*\subset \C^3$ be the holomorphic map
$f(\zeta)=\zeta(1,\imath,0)$ onto the ray of $\Agot_*$ spanned by
the null vector $(1,\imath,0)$. In this case, $g=0$ and $\eta(\zeta)=\zeta$ 
does not have a square root on $M$.  Thus, the flat null curve $F(\zeta)=\tfrac 1 2(\zeta^2, \imath \zeta^2,0)$ 
$(\zeta\in \C_*$) lies in the nontrivial isotopy class.  

2. {\em The catenoid:}  $M=\C_*$, $g(\zeta)=\zeta$, and $\eta(\zeta)=1/\zeta^2$.
Since $\eta$ has a square root on $M$, we are in the trivial isotopy class.  The same holds for the helicoid
which is parameterized by $\C$.

3. {\em Henneberg's surface} (see the formulas in L\'opez and Mart\'in's survey paper \cite{LopezMatin1999}):  
\[
	M=\C\setminus\{0,1,-1,\imath,-\imath \}, \quad g(\zeta)=\zeta, \quad \eta(\zeta)=1-\zeta^{-4}.
\]  
On a small punctured disc centered at one of the points $1$, $-1$, $\imath$, or $-\imath$, the function 
$\eta$ does not have a square root,  so we are in the nontrivial isotopy class.  
On the punctured disc $\mathbb D_*=\D \setminus\{0\}$, $\eta$ has a square root, so we are in the trivial isotopy class.
Here, $\D=\{\zeta\in \C\colon |\zeta|<1\}$.

4. {\em Two-sheeted covering of Meeks's minimal M\"obius strip} (see Meeks \cite[Theorem 2]{Meeks1981DMJ}):  
\[
	M=\C_*,\quad g(\zeta)=\zeta^2\,\dfrac{\zeta+1}{\zeta-1}, \quad \eta(\zeta)=\imath\, \dfrac{(\zeta-1)^2}{\zeta^4}.
\]
Note that  $\eta$ has a square root on $M$.  Despite the pole of $g$ at $1$, we get a holomorphic factorization 
through $\pi$ and we are in the trivial isotopy class. Let $F=u+\imath v\colon \C_* \to \C^3$ be the null curve
with this Weierstrass data. Then $u$ is invariant with respect to the fixed-point-free antiholomorphic involution 
$\Igot(\zeta) = -1/\bar \zeta$ on $\C_*$, and hence it induces a conformal minimal immersion 
$\tilde u\colon \C_*/\Igot \to\R^3$. This is Meeks's complete minimal M\"obius strip in $\R^3$ with finite total curvature
$-6\pi$.
\end{example}


\section{Preliminaries}\label{sec:prelim}

Let $M$ be an open Riemann surface. As we have already mentioned in the Introduction,
a smooth immersion $u=(u_1,\ldots,u_n)\colon M\to \R^n$ is conformal and minimal if and only if its $(1,0)$-differential
$\di u=(\di u_1,\ldots,\di u_n)$ is a $\C^n$-valued holomorphic $1$-form satisfying the nullity condition \eqref{eq:nullity}.
The {\em conjugate differential} of the map $u$ is defined by  
\[
	d^c u= \imath(\dibar u - \di u) = \Im (2 \di u).  
\]
We have
\[
	2\di u = du + \imath d^c u,\qquad dd^c u= 2\imath\,  \di\dibar u = \Delta u\cdotp  dx\wedge dy,
\]
where the second formula holds in any local holomorphic coordinate $x+\imath y$ on $M$.
Thus, $u$ is harmonic if and only if $d^c u$ is a closed $1$-form on $M$,
and in such case the equation $d^c u=dv$ holds for any local harmonic conjugate $v$ of $u$. 

Let $H_1(M;\Z)$ denote the first homology group of $M$. 
The {\em flux} of a harmonic immersion $u\colon M\to \R^n$ is the group homomorphism 
$\Flux(u) \colon H_1(M;\Z)\to\r^n$ which is defined on any homology class $[\gamma]\in H_1(M;\z)$ by
\begin{equation} \label{eq:flux}
	\Flux(u)([\gamma])=\int_\gamma d^c u = \int_\gamma \Im (2\di u).
\end{equation}
Since the integrals are independent of the choice of a path in a given homology class, 
we shall simply write $\Flux(u)(\gamma)=\Flux(u)([\gamma])$ in the sequel.

A compact set $K$ in an open Riemann surface $M$ is said to be $\Oscr(M)$-convex
if for every point $x \in M\setminus K$ there is a holomorphic function $f\in\Oscr(M)$
such that $|f(x)|>\sup_K |f|$. It is classical that this holds if and only if $M\setminus K$ does not
have any relatively compact connected components, and in this case, the Runge 
approximation theorem holds on $K$ \cite{Runge1885AM,Mergelyan1952}. 
For this reason, we also say that such a set $K$ is {\em Runge} in $M$.

We recall the notion of an {\em admissible set} (cf.\ \cite[Definition 5.1]{AlarconForstnericLopez2015MZ}).

\begin{definition}
\label{def:admissible}
A compact subset $S$ of an open Riemann surface $M$ is said to be {\em admissible} if $S=K\cup \Gamma$,
where $K=\cup_j K_j$ is a union of finitely many pairwise disjoint, compact, smoothly bounded domains 
$K_j$ in $M$ and $\Gamma=\cup_i \Gamma_i$ is a union of finitely many pairwise disjoint smooth arcs or closed 
curves that intersect $K$ only in their endpoints (or not at all), and such that their intersections with the boundary 
$bK$ are transverse. 
\end{definition}

Note that an admissible set $S\subset M$ is Runge in $M$ (i.e.,\ $\Oscr(M)$-convex)
if and only if the inclusion map $S\hra M$ induces an injective homomorphism 
$H_1(S;\z)\hra H_1(M;\z)$ of the first homology groups. If this holds, then we have the 
classical Mergelyan approximation theorem (cf.\ \cite{Mergelyan1952}): 
continuous functions on $S$ that are holomorphic in the interior $\mathring S$
can be approximated uniformly on $S$ by functions holomorphic on $M$.

Recall that $\Agot=\Agot^{n-1}\subset\C^n$ denotes the null quadric (\ref{eq:Agot}) and $\Agot_*=\Agot\setminus\{0\}$.
Given an admissible set $S=K\cup \Gamma\subset M$, we denote by 
\[
	\Ogot(S,\Agot_*)
\] 
the set of all smooth maps $S\to\Agot_*$  which are holomorphic on an unspecified open neighborhood of $K$ 
(depending on the map). We say that a map $f\in\Ogot(S,\Agot_*)$ is {\em nonflat} if it maps no component of $K$ 
and no component of $\Gamma$ to a ray in the null quadric $\Agot_*$. 
We denote by $\Ogot_*(S,\Agot_*)$ the subset of $\Ogot(S,\Agot_*)$ consisting of all nonflat maps.

%
%
%
%
\begin{definition}\label{def:generalized}
Let $M$ be an open Riemann surface, $\theta$ be a nowhere vanishing holomorphic 1-form on $M$,
and $S=K\cup \Gamma \subset M$ be an admissible subset 
(Definition \ref{def:admissible}). A {\em generalized conformal minimal immersion} on $S$ 
is a pair $(u,f\theta)$, where $f\in \Ogot(S,\Agot_*)$
and $u\colon S\to\r^n$ is a smooth map which is a conformal minimal immersion on an open 
neighborhood of $K$, such that the following properties hold:
\begin{itemize}
\item $f\theta = 2\partial u$ on an open neighborhood of $K$ in $M$, and
\item on any smooth curve $\alpha$  parameterizing a connected component of $\Gamma$ we have
that $\Re(\alpha^*(f\theta)) = \alpha^*(du)= d(u\circ\alpha)$.
\end{itemize}
A generalized conformal minimal immersion $(u,f\theta)$ is {\em nonflat} 
if $f\in\Ogot(S,\Agot_*)$ is nonflat.
\end{definition}

Note that a generalized conformal minimal immersion on a curve $C$ is nothing else than
a $1$-jet of a conformal immersion along $C$. 

The flux homomorphism can also be defined for a generalized conformal minimal immersion
$\tilde u=(u,f\theta)$ on an admissible set $S$: for each closed curve $\gamma \subset S$ we take
\[
	\Flux({\tilde u})(\gamma) = \int_\gamma \Im(f\theta).
\]
Clearly, this notion coincides with \eqref{eq:flux}  for curves $\gamma$ contained in the interior of $S$.

We denote by $\GM(S,\R^n)$ the set of all generalized conformal minimal immersions $S\to\r^n$ 
and by 
\[
	\GM_*(S,\R^n)\subset \GM(S,\R^n) 
\]
the subset consisting of all nonflat ones.


When $K$ is a compact set in a complex manifold $X$, we shall say that 
a map $f\colon K\to Y$ to another complex manifold $Y$ is {\em holomorphic on $K$} if $f$ extends
to a holomorphic map $U\to Y$ on an open neighborhood $U\subset X$ of $K$ (depending on the map).
Two such maps are considered the same if they agree on some neighborhood of $K$. 
For maps in a continuous family $f_p\colon K\to Y$ ($p\in P$), we shall assume that the neighborhood $U$
is independent of $p$. 


\section{An application of Gromov's convex integration lemma}\label{sec:CI-lemma}

Given a closed, embedded, real analytic curve $C$ in a Riemann surface $M$,
there are an open neighborhood $W \subset M$ of $C$ and a biholomorphic map 
$z\colon W \to \Omega$ onto an annulus $\Omega=\{z\in \c\colon r^{-1} < |z|< r\}$
taking $C$ onto the circle $\mathbb{S}^1=\{z\in \c: |z|=1\}$.  
The exponential map $\C \ni \zeta= x+\imath y \mapsto \exp(2\pi \imath\, \zeta)\in \C_*$ 
provides a universal covering of $\Omega$ by  
$\Sigma =\{x+\imath y :  x\in \R,\  |y|< (2\pi)^{-1}\log r \} \subset \c$, 
mapping $\R=\{y=0\}$ onto $\mathbb{S}^1$. 
We take $\zeta=x+\imath y$ as a uniformizing coordinate on $W$, with $C=\{y=0\}$. 

Let $n\ge 3$ be an integer. Assume that $u\colon M\to\r^n$ is a conformal harmonic immersion. 
The restriction $u|_W$ is given in this coordinate by a $1$-periodic conformal harmonic immersion 
$U\colon\Sigma\to \R^n$. Along the line $y=0$, we have that
\[
	U(x+\imath y) = h(x) - g(x) y + O(y^2)
\]
where $h(x)=U(x+\imath 0)$ and $g(x)=-U_y(x+\imath 0)$ are smooth $1$-periodic maps $\r\to\r^n$
and the remainder term satisfies $|O(y^2)|\le cy^2$ for some constant $c>0$ independent of $x\in\R$. 
It follows that
\[
	2\frac{\di}{\di \zeta} U(x+\imath y)|_{y=0} = \left(U_x - \imath U_y\right)|_{y=0} = 
	h'(x)+\imath  g(x).
\]
Conformality of the map $U$ implies that 
\begin{equation}\label{eq:conformal2}
	g(x) \cdotp  h'(x)=0 \quad \text{and} \quad |g(x)| = |h'(x)|>0
	\quad \text{hold\ for\ all} \ x\in \r.
\end{equation}
We also have that $d^c U =\imath(\dibar u - \di u) = -U_y dx+U_x dy$, and hence 
\[
	\int_C d^c u= \int_0^1  d^c U= - \int_0^1 U_y(x+\imath 0)\, dx = \int_0^1 g(x) \, dx.
\] 
Condition (\ref{eq:conformal2}) means that the $1$-periodic map $\sigma = h'+\imath g\colon \r \to \Agot_*$  
is a loop in  the null quadric (\ref{eq:Agot}) whose real part has vanishing period: 
\[
	\int_0^1 \Re (\sigma(x))\, dx=\int_0^1 h'(x) \, dx =0.
\] 
In \cite[Section 3]{AlarconForstnericCrelle} it is proved that for every such map $\sigma$ 
and any vector $w\in\r^n$ there exists an isotopy of $1$-periodic maps 
$\sigma_t\colon \r \to \Agot_*$ $(t\in [0,1])$ such that $\sigma_0=\sigma$ and
\begin{equation}\label{eq:isotopy}
	\int_0^1 \Re (\sigma_t(x)) \, dx=0\ \ \text{for\ all}\ t\in [0,1], \qquad
	\int_0^1 \Im (\sigma_1(x)) \, dx=w.	
\end{equation}
In particular, we can obtain $\int_0^1 \Im( \sigma_1(x)) \, dx=0$. (See \cite[Lemmas 3.2 and 3.4]{AlarconForstnericCrelle}.) 

The analogous description and results hold if $C$ is a real analytic arc. In this case, we may omit the periodicity condition
on $\sigma_t$ and replace the real period vanishing condition in \eqref{eq:isotopy} by the condition 
that $\int_0^1 \Re (\sigma_t(x)) \, dx = v_t$, where  $v_t\in \R^n$ $(t\in [0,1])$ is a given path.

We shall prove the following lemma which generalizes the above mentioned result from 
\cite{AlarconForstnericCrelle} to the parametric case, i.e., for a family of paths $\sigma_p\colon [0,1]\to \Agot_*$ 
depending continuously on a parameter $p\in P$ in a compact Hausdorff space. 
For technical reasons, we replace  \eqref{eq:isotopy} by approximate conditions; 
see property (iii) in Lemma \ref{lem:CI}.
We can also get the exact condition provided that each path $\sigma_p\colon [0,1]\to \Agot_*$
in the given family is nonflat, i.e., its image is not contained in a complex ray of $\Agot$.
We shall arrange the latter condition to hold when applying Lemma \ref{lem:CI} in the proof of Theorem \ref{th:WHE1}
given in the following section.


\begin{lemma}\label{lem:CI}
Let $Q\subset P$ be compact Hausdorff spaces and let $\sigma\colon P\times [0,1]\to \Agot_*$
be a continuous map. Consider $\sigma_p=\sigma(p,\cdotp) \colon [0,1]\to \Agot_*$ 
as a family of paths in $\Agot_*$ depending continuously on the parameter $p\in P$ . Set
\begin{equation}
\label{eq:alpha_p}
	\alpha_p =  \int_0^1 \sigma_p(s)\, ds \in \C^n, \quad  p\in P.
\end{equation}
Given a continuous family $\alpha^t_p \in \C^n$ $(p\in P,\ t\in [0,1])$ such that
\begin{equation}
\label{eq:alpha_tp}
	\text{$\alpha^t_p = \alpha_p$\ \ for all $(p,t) \in (P\times \{0\}) \cup (Q\times [0,1])$},
\end{equation}
there exists for every $\epsilon>0$ a homotopy of paths $\sigma^t_p\colon [0,1]\to \Agot_*$
$(p\in P,\ t\in [0,1])$ satisfying the following conditions:
\begin{itemize}
\item[\rm (i)]     $\sigma^t_p=\sigma_p$ \ for all $(p,t) \in (P\times \{0\}) \cup (Q\times [0,1])$; 
\vspace{1mm}
\item[\rm (ii)]   $\sigma^t_p(0)=\sigma_p(0)$ and  $\sigma^t_p(1)=\sigma_p(1)$ \ for all $p\in P$ and $t\in [0,1]$;
\vspace{1mm}
\item[\rm (iii)]  $\left| \int_0^1 \sigma^t_p(s)\, ds -\alpha^t_p\right| < \epsilon$ \ for all $p\in P$ and $t\in [0,1]$.
\end{itemize}    
\end{lemma}

\begin{proof}
We shall reduce the proof to a parametric version of Gromov's one-dimensional 
{\em convex integration lemma} (see \cite[2.1.7.\ One-Dimensional Lemma]{Gromov1973}).

By \cite[Lemma 3.1]{AlarconForstneric2014IM}, the convex hull $\Co (\Agot)\subset \C^n$ of the null quadric 
$\Agot\subset \C^n$ \eqref{eq:Agot} equals $\C^n$. Hence, there exists a number $r_1>0$ such that
\[
	\{\alpha^t_p: p\in P,\ t\in [0,1]\} \subset \Co(\Agot\cap r_1\B^n).
\]
Here, $\B^n$ is the open unit ball in $\C^n$; we sometimes omit the superscript.
By increasing $r_1$ if necessary and choosing $r_0\in (0,r_1)$ small enough, we can ensure that
\[
	\sigma(P\times [0,1]) \subset \Agot_{r_0,r_1}:= \Agot\cap r_1\B\setminus r_0\overline\B.
\]
Note that $\Co(\Agot_{r_0,r_1})=\Co( \Agot\cap r_1\B)$.  For any $\delta>0$, set
\[
	\Omega_{\delta}=\bigl\{(z_1,\ldots,z_n)\in \C^n :  |z_1^2+\cdots + z_n^2|<\delta,\ 
	r_0 <  |z_1|^2+\cdots +|z_n|^2 <r_1\bigr\}.
\]
Then we have that
\[
	\sigma(P\times [0,1])\subset  \Agot_{r_0,r_1} \subset \Omega_{\delta}\quad \text{and}\quad 
	\{\alpha^t_p : p\in P,\ t\in [0,1]\} \subset \Co(\Omega_{\delta}).
\]

Let $\rho\colon U\to \Agot_*$ be a smooth retraction from an open neighborhood $U \subset \C^n$
of $\Agot_*$. Pick $\delta>0$ small enough such that
$\Omega_\delta\subset U$ and for every path $\gamma\colon [0,1]\to \Omega_\delta$ we have
\begin{equation}\label{eq:estimate1}
	\left| \int_0^1 \gamma(s) \, ds - \int_0^1 \rho(\gamma(s))\,ds \right| < \epsilon/4.
\end{equation}
We fix a number $\delta$ satisfying these conditions.

Our construction proceeds in three steps. In the first step, we shall find a continuous family of paths 
\[
	\tau_p^t \colon [0,1] \to \Co(\Omega_\delta), \quad 	(p,t)\in P\times [0,1]
\]
satisfying conditions (i)--(iii) in Lemma \ref{lem:CI}, with $\tau^t_p$ in place of $\sigma^t_p$
and with $\epsilon$ replaced by $\epsilon/2$.
In the second step, we shall deform the family $\tau^t_p$ 
to a family of paths $\tilde \tau^t_p\colon [0,1]\to \Omega_\delta$, 
keeping the deformation fixed on the set 
\begin{equation}\label{eq:R}
	R:=(P\times \{0\}) \cup (Q\times [0,1]),
\end{equation}
such that the integrals change by less than $\epsilon/4$.
It is here that Gromov's convex integration lemma will be used.
In the last step, we shall apply to $\tilde \tau^t_p$ the retraction $\rho\colon \Omega_\delta \to \Agot_*$
in order to get the desired family of paths $\sigma^t_p=\rho\circ \tilde \tau^t_p\colon [0,1]\to \Agot_*$.

Choose a continuous function $\chi\colon P\times [0,1]\to [0,1]$ supported in a 
small neighborhood of the set $R$ \eqref{eq:R}
and satisfying $\chi=1$ on $R$. Consider the continuous family of paths 
\[
	\tilde\sigma_p^t  \colon [0,1] \to \Co(\Omega_\delta), \quad (p,t)\in P\times [0,1]
\]
given by
\begin{equation}\label{eq:tildesigma}
	\tilde\sigma_p^t(s) = \chi(p,t)\sigma_p(s) + (1-\chi(p,t)) \alpha^t_p,\quad s\in [0,1].
\end{equation}
Clearly, $\tilde\sigma_p^t=\sigma_p$ for $(p,t) \in R$ (see \eqref{eq:R})  and hence condition (i) holds. 
By \eqref{eq:alpha_p}, we have 
\[
 	\int_0^1 \tilde\sigma_p^t(s) \, ds= \chi(p,t) \alpha_p +  (1-\chi(p,t)) \alpha^t_p.
\]
Since $\chi=1$ on the set $R$ given by \eqref{eq:R} and 
$\alpha^t_p = \alpha_p$ for all $(p,t) \in R$ in view of \eqref{eq:alpha_tp}, 
the above expression is arbitrarily close to $\alpha_p^t$ uniformly on $(p,t)\in P\times [0,1]$, 
provided that the support of $\chi$ is contained in a sufficiently small 
neighborhood of $R$. Thus, we may arrange by a suitable choice of the cut-off function $\chi$ 
that 
\begin{equation}\label{eq:estimate2}
	\left| \int_0^1 \tilde \sigma^t_p(s)\, ds -\alpha^t_p\right| < \epsilon/4, \quad
	(p,t)\in P\times [0,1].
\end{equation}

In order to obtain condition (ii), we make a small correction near the two endpoints 
$s=0$ and $s=1$ of the parameter interval $[0,1]$ of our paths. 
Pick a number $0<\eta<1/3$ close to $0$ and define the path 
\begin{equation}\label{eq:tautp}
	\tau^t_p\colon [0,1]\to \Co(\Omega_\delta),\quad  (p,t)\in P\times [0,1]
\end{equation}
as follows:
\begin{itemize}
\item on $s\in [0,\eta]$ we follow the straight line segment from $\sigma_p(0)$ to $\tilde\sigma^t_p(0)$;
\vspace{1mm}
\item on $s\in [\eta,1-\eta]$ we follow the path $\tilde \sigma^t_p$ (see \eqref{eq:tildesigma}) reparameterized to
this interval. Explicitly, we take $\tau^t_p(s)=\tilde \sigma^t_p((s-\eta)/(1-2\eta))$;
\vspace{1mm}
\item  on $s\in [1-\eta,1]$ we follow the straight line segment from $\tilde\sigma^t_p(1)$ to $\sigma_p(1)$.
\end{itemize}
Clearly, the path $\tau^t_p$ satisfies conditions (i) and (ii) (with $\tau^t_p$ in place of $\sigma^t_p$).
Since the above modification does not change the integrals  very much, we have that 
\begin{equation}\label{eq:estimate3}
	\left| \int_0^1 \tau^t_p(s)\, ds - \int_0^1 \tilde \sigma^t_p\,ds \right| < \epsilon/4, \quad
	(p,t)\in P\times [0,1]
\end{equation}
provided that the number $\eta>0$ is chosen small enough.

Our next goal is to deform the family of paths $\tau^t_p$ with range in $\Co(\Omega_\delta)$ (see \eqref{eq:tautp})
to a family of paths $\sigma^t_p\colon [0,1]\to\Omega_\delta$ and with approximately the same integrals, 
keeping the deformation fixed on the set
\begin{equation}\label{eq:K}
	K := (P\times [0,1]_t \times \{0,1\}_s) \cup (R \times [0,1]_s) \subset P\times [0,1]_t \times [0,1]_s.
\end{equation}
(Recall that $R\subset P\times [0,1]$ is the set given by \eqref{eq:R}. The subscripts $s$ and $t$ 
in \eqref{eq:K} are included to help the reader understand which variable do we have in mind.)
This will be accomplished by applying a parametric version of Gromov's 
{\em one-dimensional convex integration lemma} (see \cite[Lemma 2.1.7, p.\ 337]{Gromov1973}).
The precise result that we shall use is given by \cite[Theorem 3.4, p.\ 39]{SpringBook}
where the reader can find a complete proof.

Consider the family of $\Cscr^1$ paths $f^t_p\colon [0,1]\to \C^n$ $(p\in P,\ t\in [0,1])$ given by 
\[
	f^t_p(s) = \int_0^s \tau_p^t(\xi) \, d\xi,\quad s\in [0,1].
\]
Set 
\[
	\beta_p^t=\sigma_p\colon [0,1]\to  \Agot_{r_0,r_1} \subset \Omega_\delta,
	\quad (p,t)\in P\times [0,1].
\]
The family $(f^t_p,\beta_p^t)$ $(p\in P,\ t\in [0,1])$ satisfies the hypotheses of 
\cite[Theorem 3.4, p.\ 39]{SpringBook}. Indeed, for every $(p,t)\in P\times [0,1]$ we have that
\[
	\di_s f^t_p(s) = \begin{cases}  \tau_p^t(s) \in \Co(\Omega_\delta), &  s\in [0,1]; \\
				\sigma_p(s)\in \Omega_\delta, & s\in \{0,1\}. 
				\end{cases}
\]
(See also Remark \ref{rem:Spring}.)
Applying the cited theorem to the family $(f^t_p,\beta_p^t)$, keeping 
the deformation fixed on the set $K$ given by \eqref{eq:K}, furnishes a continuous family of paths 
\[
	\tilde\tau_p^t\colon [0,1]\to \Omega_\delta,\quad (p,t)\in P\times [0,1]
\]
satisfying conditions (i), (ii) and also the following estimate:
\begin{equation}\label{eq:estimate4}
	\left| \int_0^1 \tilde\tau^t_p(s)\, ds - \int_0^1 \tau^t_p(s)\, ds \right| < \epsilon/4, \quad
	(p,t)\in P\times [0,1].
\end{equation}

Recall that $\rho \colon \Omega_\delta \to \Agot_*$ is the retraction chosen at the
beginning of the proof. Set
\[
	\sigma_p^t= \rho \circ \tilde\tau_p^t,\quad (p,t)\in P\times [0,1].
\]  
Clearly, $\sigma^t_p(s)=\tilde\tau_p^t(s)=\sigma_p(s)$ for $(p,t,s)\in K$ \eqref{eq:K} since these values 
belong to $\Agot_*$. Thus, $\sigma^t_p$ satisfies conditions (i) and (ii).
Furthermore, the estimate \eqref{eq:estimate1} shows that
\begin{equation}\label{eq:estimate5}
	\left| \int_0^1 \sigma^t_p(s)\, ds - \int_0^1 \tilde\tau^t_p(s)\, ds \right| < \epsilon/4, \quad
	(p,t)\in P\times [0,1].
\end{equation}
By combining the estimates \eqref{eq:estimate2}, \eqref{eq:estimate3}, \eqref{eq:estimate4}
and  \eqref{eq:estimate5} we see that $\sigma_p^t$ also satisfies condition (iii). 
This proves Lemma \ref{lem:CI}. 
\end{proof}

%
%
\begin{remark}[On pairs of parameter sets $Q\subset P$ in Lemma \ref{lem:CI}]  \label{rem:Spring}
We wish to clarify a certain technical point concerning the use of \cite[Theorem 3.4, p.\ 39]{SpringBook}
in the proof of Lemma \ref{lem:CI}. 

In the cited theorem, the author assumes that the parameter sets 
$C$ and $B=C\times [0,1]$ (in his notation) are compact $\Cscr^1$ manifolds.
(Note that $C$ corresponds to our parameter set $P$, while the second factor 
$[0,1]$ is the parameter interval of  the paths $f_p\colon [0,1]\to \C^n$ for $p\in C$.
In  \cite[Theorem 3.4]{SpringBook} the letter $t\in [0,1]$ is used for this parameter;
here we used the letter $s$, reserving $t\in[0,1]$ for the parameter of the homotopy.) 
An inspection of the proof
shows that the assumption of $C$ being $\Cscr^1$ manifolds
can be avoided in our situation. Indeed, the only place in \cite[proof of Theorem 3.4]{SpringBook}
where this assumption is used is on top of page 41 where the author chooses a certain 
cut-off function $\lambda$ of class $\Cscr^1$, 
supported in a small neighborhood of a certain compact set $K$ in $B=C\times [0,1]$ 
and agreeing with $1$ on a smaller neighborhood of $K$. 
(The role of $K$ is that the homotopy should be fixed in a small neighborhood of $K$.
In our proof of Lemma \ref{lem:CI}, $K$ corresponds to the set 
$(P\times \{0,1\})\cup (Q\times [0,1]) \subset P\times [0,1]$.
By including both copies $P\times \{0\}$ and $P\times \{1\}$, we ensure that
the deformations of paths made in the process are fixed near the endpoints
$s=0,1$ of the parameter interval. Furthermore, by also including the parameter 
$t\in [0,1]$ of the homotopy into the picture, the relevant set $K\subset P\times [0,1]_t\times [0,1]_s$ 
is given by \eqref{eq:K}, where $R$ is the set \eqref{eq:R}.)
Our point is that the cut-off function $\lambda$ used in  \cite[proof of Theorem 3.4]{SpringBook}
only needs to be of class $\Cscr^1$ in the $s$ variable (the parameter on the path) 
and its derivative $\di_s \lambda$ must be continuous in all variables $(p,t,s)$. 
Such $\lambda$ obviously exists in our situation if $P$ is an arbitrary compact Hausdorff space.
\qed \end{remark}


\section{Parametric h-principle for the inclusion $\Re \Ngot_*(M,\C^n) \hra \Mgot_*(M,\R^n)$}
\label{sec:PHP1}

In this section, we prove the parametric h-principle with approximation for the inclusion 
\[
	\Re \Ngot_*(M,\C^n)= \{u\in \Mgot_*(M,\R^n):\Flux(u)=0\}  \longhookrightarrow \Mgot_*(M,\R^n)
\]
considered in Theorem \ref{th:WHE1}.

\begin{theorem} \label{th:PHP1}
Assume that $M$ is an open Riemann surface, $Q\subset P$ are compact Hausdorff spaces,  
$D\Subset M$ is a smoothly bounded domain whose closure $\bar D$ is $\Oscr(M)$-convex, 
and $u\colon M\times P\to\R^n$ $(n\ge 3)$ is a continuous map satisfying the following conditions:
\begin{itemize}
\item[\rm (a)] $u_p=u(\cdotp,p)\colon M\to \R^n$ is a nonflat conformal minimal immersion for every $p\in P$;
\item[\rm (b)] $u_p|_{\bar D} \colon \bar D \to \R^n$ has vanishing flux for every $p\in P$;
\vspace{1mm}
\item[\rm (c)] $\Flux(u_p)=0$ for every $p\in Q$.
\end{itemize}
Given a number $\epsilon>0$, there exists a homotopy $u^t \colon M\times P\to \R^n$ 
$(t\in [0,1])$ such that the map $u^t_p := u^t(\cdotp,p)\colon M\to \R^n$ is a nonflat conformal 
minimal immersion for every $(p,t)\in P\times [0,1]$ satisfying the following conditions:
\begin{itemize}
\item[\rm (1)] $u^t_p = u_p$ \ for every $(p,t)\in (P\times \{0\}) \cup (Q\times [0,1])$; 
\vspace{1mm}
\item[\rm (2)]   $|u^t_p(x) - u_p(x)|<\epsilon$ for all $x\in \bar D$ and $(p,t) \in P\times [0,1]$;
\vspace{1mm}
\item[\rm (3)]   $u^t_p|_{\bar D}$ has vanishing flux for every $(p,t) \in P\times [0,1]$;
\vspace{1mm}
\item[\rm (4)]   $\Flux(u^1_p)=0$ for every $p\in P$.
\end{itemize}
\end{theorem}

Let us explain the content of the theorem in nontechnical terms.
The assumption (b) is equivalent to saying that  the restriction $u_p|_{\bar D}$  is the real part of a 
holomorphic null curve $\bar D \to\C^n$ for every $p\in P$, while (c) 
says that for every $p\in Q$, $u_p\colon M\to\R^n$ is the real part of a globally defined 
holomorphic null curve $F_p\colon M\to\C^n$. The conclusion of the theorem is that we can
deform $u$ through a homotopy $u^t\colon M\times P\to\R^n$  $(t\in [0,1])$ consisting of 
nonflat conformal minimal immersions $u^t_p := u^t(\cdotp,p)\colon M\to \R^n$ such that the
homotopy is fixed for $p\in Q$ (condition (1)), it approximates $u$ uniformly on $\bar D\times P$
for every $t\in [0,1]$ (condition (2)), the restrictions $u^t_p|_{\bar D}$ are real parts of null curves $\bar D\to\R^n$
(condition (3)), and at $t=1$ the family $u^1_p\colon M\to\R^n$ consists of real parts
of null curves $F_p\colon M\to \C^n$ (condition (4)). 

Assuming for a moment that Theorem \ref{th:PHP1} holds, we now give:

\begin{proof}[Proof of Theorem \ref{th:WHE1}]
Let $k\in \Z_+$. Applying Theorem \ref{th:PHP1} with $P=\S^k$ (the real $k$-sphere) 
and $Q=\emptyset$ shows that the map
\begin{equation}
\label{eq:pi_k}
	\pi_k(\Re \Ngot_*(M,\C^n)) \longrightarrow \pi_k( \Mgot_*(M,\R^n)),
\end{equation}
induced by the inclusion, is surjective. Applying Theorem \ref{th:PHP1} with $P=\overline\B^{k+1}$ (the closed ball in $\R^{k+1}$)
and $Q=\S^k=b\B^{k+1}$ shows that the map  \eqref{eq:pi_k}
is also injective. 
\end{proof}

\begin{proof}[Proof of Theorem \ref{th:PHP1}]
Pick a smooth strongly subharmonic Morse exhaustion function $\rho\colon M\to \R$ and exhaust 
$M$ by sublevel sets 
\[
	D_j=\{x\in M\colon \rho(x)<c_j\}, \quad j\in\N
\]
where $c_1<c_2<c_3<\ldots$ is an 
increasing sequence of regular values of $\rho$ such that $\lim_{j\to\infty} c_j=\infty$. 
We may assume in addition that $D_1=D$ is the domain  in Theorem \ref{th:PHP1} 
and each interval $[c_j,c_{j+1}]$ contains at most one critical value of the function $\rho$. 
Let $\epsilon>0$ be as in the theorem.
Pick a sequence $\epsilon_j >0$ with $\sum_{j=1}^\infty \epsilon_j<\epsilon$. Set
\[
	u^t_{p,1} := u_p|_{\bar D_1}, \quad (p,t)\in P\times [0,1].
\]
We shall recursively construct a sequence of homotopies of conformal minimal immersions
\[
	u^t_{p,j} \colon \bar D_j\lra \R^n,\quad (p,t)\in P\times[0,1],\ \  j\in\N
\]	
satisfying the following conditions for every $j=2,3,\ldots$:
\begin{itemize}
%
%
\item[$(a_j)$] $u^t_{p,j} = u_p|_{\bar D_{j}}$ for every $(p,t)\in (P\times \{0\}) \cup (Q\times [0,1])$; 
\vspace{1mm}
\item[$(b_j)$]  $\|u^{t}_{p,j} - u^t_{p,j-1}\|_{\bar D_{j-1}} <\epsilon_j$ for all $(p,t)\in P\times [0,1]$;
\vspace{1mm}
\item[$(c_j)$]  $\Flux(u^t_{p,j}|_{\bar D_{j-1}}) = \Flux(u^t_{p,j-1})$  for every $(p,t)\in P\times [0,1]$;
\vspace{1mm}
\item[$(d_j)$]  $\Flux(u^1_{p,j})=0$ on $\bar D_j$ for every $p\in P$.
\end{itemize}
Note that condition $(d_1)$ holds by the definition of $u^t_{p,1}$, while the other conditions are vacuous for $j=1$.
Clearly, these conditions imply that the limit
\[
	u^t_p=\lim_{j\to \infty} u^t_{p,j}\colon M\lra \R^n, \quad p\in P,\ t\in [0,1]
\]
exists and satisfies the conclusion of Theorem \ref{th:PHP1}. Indeed, $(a_j)$ ensures that 
all homotopies are fixed on the parameter set $(P\times \{0\}) \cup (Q\times [0,1])$ which will give condition (1)
in the theorem. Condition $(b_j)$ ensures that the sequence converges uniformly on compacts 
in $M\times P\times [0,1]$ and the limit $u^t_p$ satisfies condition (2) in the theorem.
Condition (3) follows from $(c_j)$, and condition (4) is a consequence of $(d_j)$.

We shall now describe the recursion. We distinguish two topologically different cases:  
(a) the noncritical case, and (b) the critical case.

\vspace{1mm} 

{\bf (a) The noncritical case.} 
Let $K$ and $L$ be smoothly bounded compact domains 
in $M$ such that $\bar D\subset K\subset \mathring L$  and $K$ is a strong deformation retract of $ L$. 
In the recursive scheme, this corresponds to the case $K=\bar D_j$, $L=\bar D_{j+1}$,
and $\rho$ has no critical values in $[c_j,c_{j+1}]$. However, in the critical case considered
below we shall have to use the noncritical case also for certain noncritical pairs $K\subset  L$ 
that are not sublevel sets of this particular function $\rho$.

Let $u\colon M\times P\to\R^n$ $(n\ge 3)$ be as in Theorem \ref{th:PHP1}, satisfying conditions 
(a), (b) and (c). As before, we shall write $u_p=u(\cdotp,p)\colon M\to \R^n$ for $p\in P$.
Assume also that we are given a continuous family of nonflat conformal minimal immersions 
\begin{equation}\label{eq:utp}
	u^t_p\colon K\lra \R^n,\quad (p,t)\in P\times [0,1]
\end{equation}
satisfying the following conditions:
\begin{itemize}
\item[\rm (a')]  $u^t_p = u_p|_K$ for every $(p,t)\in (P\times \{0\}) \cup (Q\times [0,1])$; 
\vspace{1mm}
\item[\rm (b')]  $u^t_p$ has vanishing flux on $\bar D$ for every $(p,t)\in P\times [0,1]$;
\vspace{1mm}
\item[\rm (c')]  $u^1_p$ has vanishing flux on $K$ for every $p\in P$.
\end{itemize}
Given a number $\epsilon>0$, we shall find a homotopy of conformal minimal immersions 
\begin{equation}\label{eq:tildeupt}
	\tilde u^t_p\colon  L\lra \R^n,\quad (p,t)\in P\times [0,1]
\end{equation}
satisfying the following conditions:
\begin{itemize}
\item[$(\alpha)$]  $\tilde u^t_p = u_p|_L$ for every $(p,t)\in (P\times \{0\}) \cup (Q\times [0,1])$; 
\vspace{1mm}
\item[$(\beta)$]  $\|\tilde u^t_p - u^t_p\|_K <\epsilon$ for every $(p,t)\in P\times [0,1]$;
\vspace{1mm}
\item[$(\gamma)$] $\Flux(\tilde u^t_{p}|_K) = \Flux(u^t_p|_K)$ for every $(p,t)\in P\times [0,1]$.
\end{itemize}
In particular, it follows from properties (b') and $(\gamma)$ that $\tilde u^t_p$ has vanishing flux on $\bar D$ 
for every $(p,t)\in P\times [0,1]$, and $\tilde u^1_p$ has vanishing flux on $K$ for every $p\in P$
in view of properties (c') and $(\gamma)$. Since the set $K$ contains all the topology of $L$, we also get that
\begin{itemize}
\item[$(\delta)$]  $\tilde u^1_p$ has vanishing flux on $L$ for every $p\in  P$.
\end{itemize}
In the recursive scheme used in the proof of Theorem \ref{th:PHP1}, with $K_j=\bar D_j$ and $L=\bar D_{j+1}$,
the family $u^t_p=u^t_{p,j}\colon \overline D_j\to\R^n$ satisfies the inductive hypotheses $(a_j)$, $(b_j)$, $(c_j)$
and $(d_j)$, and we may take the family $u^t_{p,j+1} := \tilde u^t_{p}$ as the new homotopy
of conformal minimal immersions $L=\bar D_{j+1}\to \R^n$. Indeed, conditions
$(\alpha)$--$(\delta)$ for $\tilde u^t_{p}$ exactly correspond to the respective conditions 
$(a_{j+1})$--$(d_{j+1})$ for the next term $u^t_{p,j+1}$ in the recursion.

The construction of a family $\tilde u^t_p$ \eqref{eq:tildeupt} satisfying $(\alpha)$--$(\gamma)$ is  accomplished 
in a similar way as was done in the nonparametric case in \cite[Section 5]{AlarconForstnericLopez2015MZ}, 
but the details are somewhat more involved in the present situation involving parameters; we now explain the details.

Pick a nowhere vanishing holomorphic $1$-form $\theta$ on $M$. 
By classical results, there exists a Runge homology basis $\Bcal=\{\gamma_i\colon i=1,\ldots,l\}$ for $H_1(K;\Z)$, 
i.e., the union $|\Bcal|=\cup_{i=1}^l |\gamma_i|$ of supports of the curves $\gamma_i$ is an 
$\Oscr(K)$-convex subset of $K$, such that
the curves $\Bcal'=\{\gamma_1,\ldots,\gamma_m\}$ for some $m\in \{0,\ldots,l\}$ form a homology basis
for $H_1(\bar D;\Z)$. Note that $\Bcal$ is then also a homology basis for $H_1(L;\Z)$.
We denote by $\Pcal$ the period map associated to $\Bcal$:
\[
	\Pcal(f) =  \left(\int_{\gamma_i} f \theta \right)_{i=1,\ldots,l} \in (\C^n)^l,\qquad  f\in\Acal(K,\Agot_*).
\]
Also, $\Pcal'\colon \Acal(\bar D,\Agot_*)\to (\C^n)^m$ will denote the period map with respect to $\Bcal'$. 

Consider the continuous family of nonflat holomorphic maps 
\[
	f^t_p:=2\di u^t_p/\theta \colon K \longrightarrow \Agot_*,\quad p\in P,\ t\in [0,1].
\] 
Conditions $(\alpha)$--$(\gamma)$ on $u^t_p$ imply the following:
\begin{eqnarray*}
	\Re \Pcal(f^t_p)                  &=& 0,\quad (p,t)\in P \times [0,1];  \\
	\Pcal'(f^t_p|_{\bar D})         &=& 0,\quad (p,t)\in P \times [0,1]; \\
	\Pcal(f^1_p)                       &=&  0,\quad \,p\in P.
\end{eqnarray*}
Since $u^t_p$ is nonflat on $K$ for each $(p,t)\in P\times [0,1]$ by the assumption, 
we can apply \cite[Lemma 5.1]{AlarconForstneric2014IM}  
(see also \cite[Lemma 3.6]{AlarconForstnericCrelle} for the parametric case)
to embed the family $f^t_p$ as the core $f^t_p=f^t_{p,0}$ of a {\em period dominating spray} of holomorphic  maps 
\[
	f^t_{p,\zeta} \colon K \lra \Agot_*, \quad \zeta\in B,\ p\in P,\ t\in [0,1]
\]
depending holomorphically on a parameter $\zeta=(\zeta_1,\ldots,\zeta_N)$ 
in a ball $0\in B \subset \C^N$ for some $N\in \N$.  The period domination property means that the period map 
\begin{equation}\label{eq:period}
	B \ni \zeta\longmapsto \Pcal(f^t_{p,\zeta} )=
	\left(\int_{\gamma_i} f^t_{p,\zeta}  \theta \right)_{i=1,\ldots,l} \in (\C^n)^l,
\end{equation}
associated to the homology basis $\Bcal=\{\gamma_i\}_{i=1,\ldots,l}$ of $H_1(K;\Z)$, is submersive
at $\zeta=0$, i.e., its differential at $\zeta=0$ is surjective for every $(p,t)\in P\times [0,1]$. 

Since $\Agot_*$ is an Oka manifold and $K$ is a deformation retract of $ L$, 
the parametric Oka property with approximation (see \cite[Theorem 5.4.4, p.\ 193]{Forstneric2011})
allows us to approximate the spray $f^t_{p,\zeta}\colon K\to \Agot_*$ 
uniformly on $K$ and uniformly with respect to the parameters $(p,t,\zeta)$ 
(allowing the $\zeta$-ball $B\subset\C^N$ to shrink a little) by a holomorphic spray 
\[
	g^t_{p,\zeta}\colon  L \lra \Agot_*,\quad (p, t) \in P\times [0,1],\ \zeta\in rB
\]
for some $r\in (1/2,1)$. More precisely, $g^t_{p,\zeta}$ is holomorphic on (a neighborhood of) $L$ 
and in $\zeta\in rB$, and it is continuous with respect to $(p,t)\in P\times [0,1]$. 
In the cited theorem, the parameter spaces $Q\subset P$ are Euclidean compacts.
However, since $\Agot_*$ is an elliptic manifold in the sense of Gromov
(see \cite[Definition 5.5.11, p. 203]{Forstneric2011} and \cite[Example 4.4]{AlarconForstneric2014IM}),
we can also apply \cite[Theorem 6.2.2, p.\ 243]{Forstneric2011} that holds for every pair of 
compact Hausdorff parameter spaces $Q\subset P$.

If  the approximation of $f^t_{p,\zeta}$ by $g^t_{p,\zeta}$ is sufficiently close, 
the implicit function theorem gives (in view of the period domination property of the spray 
$f^t_{p,\zeta}$) a continuous map 
\begin{equation}\label{eq:zeta}
	\zeta \colon P\times [0,1] \lra rB \subset\C^N
\end{equation}
that vanishes on the set $(p,t)\in (P\times \{0\})\cup (Q\times [0,1])$ such that the homotopy of holomorphic maps 
\[
	\tilde f^t_p := g^t_{p,\zeta(p,t)} \colon  L\lra \Agot_*,\quad 
	(p,t)\in P\times [0,1]
\] 
satisfies the following period conditions:
\[
	\Pcal(\tilde f^t_p)=\Pcal(f^t_p), \qquad (p,t)\in P \times [0,1].
\] 
We are using the period map $\Pcal$ \eqref{eq:period} for the set $K$ which is a deformation retract of $L$, so it contains all the topology of $L$. In particular, we have that
\begin{eqnarray*}
	\Re \Pcal(\tilde f^t_p) &=& 0,\quad (p,t)\in P \times [0,1];  \\
	\Pcal'(\tilde f^t_p|_{\bar D}) &=& 0,\quad (p,t)\in P \times [0,1]; \\
	\Pcal(\tilde f^1_p)      &=&  0, \quad \, p\in P.
\end{eqnarray*}
Note also that 
\begin{equation*}
	\tilde f^t_p= f^t_p, \quad (p,t)\in (P\times \{0\})\cup (Q\times [0,1])
\end{equation*}
since the function $\zeta=\zeta(p,t)$ (see \eqref{eq:zeta}) vanishes on $(P\times \{0\})\cup (Q\times [0,1])$. 

Assume that the set $K$ (and hence $L$) is connected. Choose a point $x_0\in K$ and set
\[
	 \tilde u^t_p(x) := u^t_p(x_0) + \int_{x_0}^x \Re(\tilde f^t_p\theta),\quad x\in L,\ 
	 (p,t)\in P\times [0,1].
\]
Since the $1$-form $\Re(\tilde f^t_p\theta)$ has vanishing periods for all $(p,t)\in P\times [0,1]$,
the integral is independent of the choice of the path in $L$,
and hence $\tilde u^t_p\colon L\to\R^n$ is a continuous family of conformal minimal immersions.
If $K$ is disconnected, we can apply the same argument on every connected component.
It is immediate that this family satisfies conditions $(\alpha)$--$(\gamma)$.
This closes the induction step in the noncritical case.

\vspace{1mm}

{\bf (b) The critical case.}  Now, $K$ and $L$ are smoothly bounded compact subsets of $M$
such that $\bar D \subset K\subset L$, and $L$  admits a strong deformation retraction onto a compact set 
of the form $S=K\cup E$, where $E$ is an embedded arc in the complement of $K$ 
which is attached with both endpoints to $K$. (In the recursive scheme, this case occurs when
$K=\bar D_j$, $L=\bar D_{j+1}$ and $\rho$ contains a critical point $x_0\in D_{j+1} \setminus \bar D_j$.
The arc $E$ corresponds to the stable manifold of the critical point $x_0$ with respect to the gradient flow.)
We may  assume that $E$ is real analytic and intersects $bK$ transversely at both endpoints. 

Let $u^t_p\colon K\to \R^n$ be as in \eqref{eq:utp}, satisfying conditions (a'), (b') and (c').
It suffices to construct a new homotopy $\tilde u^t_p$   \eqref{eq:tildeupt}
satisfying conditions ($\alpha$)--($\delta$) on a neighborhood of the admissible set 
$S:=K\cup E$; the noncritical case explained above then allows us to extend it from a 
suitable neighborhood of $S$ to $L$, with approximation on $S$, so that all required properties
are satisfied.

There are two topologically different cases to consider.

{\bf Case 1:}  the arc $E$ closes inside the domain $K$ to a Jordan curve $C$ such that 
$E=\overline{C\setminus K}$. This happens when the endpoints of $E$ belong to the same 
connected component of $K$. In this case, the curve $C$ is a new element of the 
homology basis for  $H_1(L;\Z)$.

\vspace{1mm}

{\bf Case 2:} the endpoints of the arc $E$ belong to different connected components of $K$.
In this case, no new element of the homology basis appears, and the number of connected 
components of the domain decreases by one.

\vspace{1mm}

The treatment of both cases is similar; we begin by considering the first one.

Thus, let $C\subset M$ be a real analytic Jordan curve which intersects
$K$ in a closed arc $C_3$.
(As pointed out in \cite{AlarconForstnericCrelle}, real analyticity of $C$ is used merely for the convenience
of exposition; one may equally well work with smooth curves.) Note that the set
\[
	S:= K\cup C = K \cup E \subset M
\] 
is admissible (cf.\ Definition \ref{def:admissible}). 

Recall that $u^0_p=u_p\colon M\to \R^n$ $(p\in P)$ is a continuous family of nonflat conformal minimal immersion.
Consider the continuous family of nonflat smooth maps 
\[
	f_p = (2\di u_p/\theta) |_S  \colon S \lra \Agot_*, \quad p\in P.
\]
Then $(u_p,f_p\theta)\in \GM_*(S)$ $(p\in P)$ is a continuous family of nonflat generalized conformal minimal 
immersions on the set $S$ (see Definition \ref{def:generalized}). 
Similarly, we introduce the family of nonflat  holomorphic maps
\begin{equation}\label{eq:sigma^t_p}
	f^t_p = (2\di u^t_p/\theta)|_K  \colon K \lra \Agot_*, \quad (p,t)\in P\times [0,1].
\end{equation}
Note that $f^0_p=f_p|_K$ for all $p\in P$.
Recall that $\Flux(u^t_p|_K)=0$ for all $(p,t)\in P\times [0,1]$ by the inductive hypothesis.

\smallskip

{\bf Claim:}  there exist continuous families of smooth maps
\begin{equation}\label{eq:handle}
	f^t_p \colon S \lra \Agot_*,\quad u^t_p\colon S \lra \R^n,  \quad (p,t) \in P\times [0,1]
\end{equation}
satisfying the following conditions:
\begin{itemize}
\item[$(\mathfrak a)$] $u^t_p$ and $f^t_p$ agree with the already given maps on $K$ for each $(p,t)\in P\times [0,1]$;
\vspace{1mm}
\item[$(\mathfrak b)$] 
$f^t_p=f_p$ and $u^t_p=u_p|_S$ for each $(p,t)\in (P\times \{0\})\cup (Q\times [0,1])$; 
\vspace{1mm}
\item[$(\mathfrak c)$]
the pair $U^t_p:=(u^t_p,f^t_p\theta)\in \GM_*(S)$ is a generalized conformal minimal immersion on $S=K\cup C$ 
for each $(p,t)\in P\times [0,1]$. In particular, 
\[
	\int_C \Re(f^t_p \theta)=  \int_C  \Re(2\di u^t_p) = \int_C du^t_p =0;
\]  
\item[$(\mathfrak d)$] 
at $t=1$, $\Flux(U^1_p)(C) = \int_C \Im(f^1_p \theta)=0$ for all $p\in P$.

\end{itemize}

Assume for a moment that a family $U^t_p = (u^t_p,f^t_p\theta)\in  \GM_*(S)$ satisfying conditions 
$(\mathfrak a)$--$(\mathfrak d)$ exists. 
We can then complete the induction step as in \cite[proof of Theorem 5.3]{AlarconForstnericLopez2015MZ}. 
Indeed, the cited result says that we can approximate the family of generalized conformal minimal immersions 
$U^t_p=(u^t_p,f^t_p\theta)$ arbitrarily closely in the $\Cscr^1(S)$ topology 
by a continuous family $\wt U^t_p = (\tilde u^t_p,2\di \tilde u^t_p)$, where 
\[
	\tilde u^t_p\colon V \lra \R^n,\quad  (p,t)\in P\times [0,1]
\] 
is a continuous family of conformal minimal immersions in an open neighborhood $V\subset M$ of $S$, such that 
\[
	\tilde u^t_p =  u^t_p \quad \text{for all}\ \ (p,t)\in (P\times \{0\})\cup (Q\times [0,1])
\]
and 
\[
	\Flux(\tilde u^t_p) = \Flux(u^t_p)\ \ \text{on $H_1(S;\Z)$ for all $(p,t)\in P\times [0,1]$.}
\]
In particular, we have that
\[
 	\Flux(\tilde u^1_p)=0\ \  \text{on $H_1(S;\Z)$ for all $p\in P$}.
\] 
We may choose a smaller neighborhood $W\Subset V$ of $K\cup E$ such that $\overline W$ is a 
strong deformation retract of $L$. Hence, this reduces Case 1 (of the critical case) 
to the noncritical case and thereby closes the induction step.

It remains to prove the Claim, i.e., to find a continuous family of pairs $U^t_p = (u^t_p,f^t_p)$ 
as in \eqref{eq:handle}.  To this end, we shall use Lemma \ref{lem:CI}.

We parameterize the closed curve $C\subset S$ by a real analytic map $\gamma\colon [0,3]\to C$ such that
$C=C_1\cup C_2 \cup C_3$, where $C_i=\gamma([i,i+1])$ for $i=0,1,2$ and 
$C_3=C\cap K$. Hence, $E=C_1\cup C_2$. 
We extend the family of maps $f^t_p\colon K\to \Agot_*$  given by \eqref{eq:sigma^t_p}
to $S=K\cup C$ so that the extension is continuous in all variables and it satisfies
properties $(\mathfrak a)$ and $(\mathfrak b)$ for $f^t_p$. Choose a small number $\eta>0$ 
and set 
\[
	I_1=[\eta,1-\eta],\quad I_2=[1+\eta,2-\eta],\quad C'_j=\gamma(I_j)\ \ \text{for $j=1,2$}.
\]
By using a smooth cut-off function in the parameter of the homotopy, 
we can arrange that  $f^t_p $ is independent of $t\in [0,1]$ on $C'_1\cup C'_2$ for each $p\in P$,
so it equals $f_p$ there. These curves are nonflat on $C'_1$ and on $C'_2$ by the assumption on $u_p$,
a property that will be used in the sequel for the construction of period dominating sprays.

Denote the parameter on $[0,3]$ by $s$. Then
\begin{equation}\label{eq:coordinate-s}
	\gamma^*(f^t_p \theta) (s) = f^t_p(\gamma(s))\, \theta(\gamma(s),\gamma'(s)) \, ds = \sigma^t_p(s) \, ds,
\end{equation}	
where the map $\sigma^t_p\colon C \to \Agot_*$ is defined by the above equation.
It suffices to explain how to modify the paths  $\sigma^t_p$ to ensure  
properties $(\mathfrak a)$--$(\mathfrak d)$ for the corresponding pairs $(u^t_p,f^t_p \theta)$. 
Explicitly, we shall modify $\sigma^t_p$ on $I_1 \cup I_2$ in order to get 
\begin{equation}\label{eq:ftp}
	\int_C \Re(f^t_p \theta) = \int_0^3 \Re (\sigma^t_p(s)) \, ds =0,
	\quad (p,t)\in P\times [0,1]
\end{equation}
(see condition $(\mathfrak c)$) and 
\begin{equation}\label{eq:f1p}
	\int_C \Im(f^1_p \theta) =  \int_0^3 \Im(\sigma^t_p(s)) \, ds =0, \quad p\in P
\end{equation}
(see condition $(\mathfrak d)$). The value of $\sigma^t_p$ on  $[0,3]\setminus I_1\cup I_2$ will not change in 
these modifications. 

On  the segment $I_1$, we apply Lemma \ref{lem:CI} to the family of paths $\{\sigma_p|_{I_1}: p\in P\}$
for a small $\epsilon>0$ to get paths 
\[
	\sigma^t_p \colon I_1 \to \Agot_*,\quad p\in P,\ t\in [0,1]
\]
which agree with $\sigma_p$ near the endpoints of $I_1$ and satisfy
\begin{equation}\label{eq:epsilon}
	\left| \int_0^3 \Re (\sigma^t_p(s)) ds  \right| < \epsilon
	\ \ \text{for $(p,t)\in P\times [0,1]$},
	\qquad \left| \int_0^3 \sigma^1_p(s) ds \right| < \epsilon \ \  \text{for $p\in P$}.
\end{equation}
Furthermore, we may take 
\[
	\sigma^t_p=\sigma_p\quad \text{for all}\ \  (p,t) \in (P\times \{0\})\cup (Q\times [0,1])
\]
since the integrals in \eqref{eq:epsilon} vanish for such values of $(p,t)$ by the assumptions.

We shall now change the integrals in \eqref{eq:epsilon} to zero by another deformation 
of our paths that is supported on the segment $I_2=[1+\eta,2-\eta]$. Since $\sigma_p$ 
is nonflat on $I_2$ for every $p\in P$, we can apply \cite[Lemma 5.1]{AlarconForstneric2014IM} 
(see also \cite[Lemma 3.6]{AlarconForstnericCrelle}) in order to embed the family of paths 
$\{\sigma_p|_{I_2}\colon I_2\to \Agot_*,\ p\in P\}$ into a period dominating spray of paths 
\[
	\tau_{p,\zeta} \colon I_2 \lra \Agot_*,\quad  p\in P,\ \zeta\in \B^N
\]
depending holomorphically on a parameter $\zeta \in \B^N\subset \C^N$ for some big $N\in\N$, such that 
\begin{itemize}
\item $\tau_{p,0}=\sigma_p$ for all $p\in P$, and 
\vspace{1mm}
\item $\tau_{p,\zeta}(s)$ is independent of  $\zeta$ for 
all $s\in I_2$ sufficiently near the endpoints of $I_2$. 
\end{itemize}

Assuming that the number $\epsilon>0$ in \eqref{eq:epsilon} was chosen small enough,
the implicit function theorem furnishes a continuous function $\zeta \colon P\times [0,1]\lra \B^N$ 
that vanishes on $(P\times \{0\})\cup (Q\times [0,1])$ such that, setting
\[
	\sigma^t_p(s)=\tau_{p,\zeta(p,t)}(s),\quad s\in I_2,\ (p,t)\in P\times [0,1]
\]
and keeping the already defined values of $\sigma^t_p(s)$ for $s\in [0,3]\setminus I_2$, we have that
\begin{eqnarray}
\label{eq:period0}
	\int_C \Re (\sigma^t_p(s)) \, ds  &=& 0 \quad \text{for $(p,t)\in P\times [0,1]$};   \\
\label{eq:period1}	
	\int_C \Im (\sigma^1_p(s) \, ds &=&   0 \quad  \text{for $p\in P$}.
\end{eqnarray}
Clearly, the associated maps $f^t_p\colon S\to \Agot_*$, given by \eqref{eq:coordinate-s},
satisfy conditions \eqref{eq:ftp} and \eqref{eq:f1p}. Finally, we define the maps $u^t_p\colon C\to \R^n$ 
for $(p,t)\in P\times [0,1]$ by setting
\[
	u^t_p(\gamma(s)) = u^t_p(\gamma(0)) + \int_0^s  \sigma^t_p(\tau) \, d\tau, \quad s\in [0,3]. 
\]
It is clear from the construction that the family $(u^t_p,f^t_p\theta)\in \GM_*(S)$ satisfies the Claim. 

The details in Case 2 (of the critical case) are similar. 
In this case, we do not have any new element of the homology basis 
for the set $L$. Choose a regular parameterization $\gamma\colon [0,1]\to E$ of the arc $E$ that is attached with 
its endpoints $x=\gamma(0)$ and $y=\gamma(1)$ to the set $K$. The period vanishing condition
\eqref{eq:period0} is now replaced by the condition
\[
	\int_E \Re (f^t_p \theta)=\int_0^1 \sigma^t_p(s)\, ds = u^t_p(y)-u^t_p(x),
\]
while the condition \eqref{eq:period1}  becomes irrelevant. 

This complete the induction step and hence the proof of Theorem \ref{th:PHP1}. 
\end{proof}

\section{Parametric h-principle for directed immersions of Riemann surfaces}
\label{sec:Aimmersions}

In this section, we prove the parametric h-principle for immersions of open Riemann surfaces
into $\C^n$ directed by certain conical subvarieties $A\subset \C^n$, in the sense 
that the derivative of the immersion belongs to $A_*=A\setminus \{0\}$; see Theorem  \ref{th:PHP2}.
A crucial condition for this result is that $A_*$ be an Oka manifold. 
The basic h-principle in this context was provided by \cite[Theorem 2.6]{AlarconForstneric2014IM}.
As a corollary, we show that the natural map 
$\Igot_{A,*}(M,\C^n) \to \Oscr(M,A_*)$ from the space of all nondegenerate holomorphic 
immersions $M\to \C^n$ directed by $A$ to the space of all holomorphic maps $M\to A_*$ 
is a weak  homotopy equivalence (see Theorem \ref{th:WHE-A}). These results apply in particular to 
null holomorphic curves in $\C^n$ and give the corresponding part of Theorem \ref{th:WHE2}.

Let $n\ge 3$. Assume that $A\subset \C^n$ is a closed conical complex subvariety of $\C^n$ such that 
$A_*=A\setminus \{0\}$ is nonsingular. By {\em conical}, we mean that $A$ is a union of complex lines passing
through the origin. Any such subvariety $A$ is algebraic by Chow's theorem. 
Without loss of generality for the results to be presented, we may assume that $A$ is irreducible, 
i.e., its regular part $A_*$ is connected, and that $A$ is not contained in any proper complex subspace of $\C^n$. 
The following is \cite[Definition 2.1]{AlarconForstneric2014IM}.

%
%
\begin{definition}\label{def:Aimmersion}
Let $M$ be an open Riemann surface, and let $\theta$ be a nowhere vanishing holomorphic $1$-form on $M$.
A holomorphic immersion $F=(F_1,\ldots, F_n)\colon M\to \C^n$ is said to be {\em directed by $A$}, 
or to be an {\em $A$-immersion}, if the holomorphic map $f=dF/\theta=(f_1,\ldots,f_n)\colon M\to \C^n$
has range in $A_*$.
\end{definition}

If $A$ is the null quadric  \eqref{eq:Agot}, then an $A$-immersion is just a holomorphic null curve. 

Observe that a holomorphic map $f\colon M\to A_*$ determines
an $A$-immersion $F\colon M\to\C^n$ with $dF=f\theta$ if and only if the holomorphic $1$-form
$f\theta$ has vanishing periods, i.e., $\int_\gamma f\theta=0$ for every closed curve $\gamma\subset M$.
In this case, $F(x)=\int^x f\theta$ for $x\in M$. 

%
%
\begin{definition}\label{def:nondegenerate}
A holomorphic map $f\colon M\to A_*$ is {\em nondegenerate}
if the tangent spaces $T_{f(x)} A \subset T_{f(x)}\C^n \cong \C^n$ over all points $x\in M$ span $\C^n$; see \cite[Definition 2.2]{AlarconForstneric2014IM}. 
If $M$ is not connected, we ask that this holds on every connected component.
An $A$-immersion $F\colon M\to \C^n$ is nondegenerate if the map $f=dF/\theta\colon M\to A_*$
is nondegenerate.  
\end{definition}


We shall prove the following parametric h-principle with approximation for directed immersions of 
Riemann surfaces with an Oka directional manifold. 

\begin{theorem} \label{th:PHP2}
Assume that $M$ is an open Riemann surface, $\theta$ is a nowhere vanishing holomorphic $1$-form on $M$,
$D\Subset M$ is a smoothly bounded domain whose compact closure $K:=\bar D$ is $\Oscr(M)$-convex, 
$A\subset \C^n$ is a closed conical complex subvariety of $\C^n$ such that $A_*=A\setminus \{0\}$ is an Oka manifold, 
and $Q\subset P$ are compact sets in a Euclidean space. 

Assume that $f \colon M\times P\to A_*$ is a continuous map satisfying the following conditions:
\begin{itemize}
\item[\rm (a)] $f_p=f(\cdotp,p)\colon M\to A_*$ is a nondegenerate holomorphic map for every $p\in P$;
\vspace{1mm}
\item[\rm (b)] the $1$-form $f_p\theta$ has vanishing periods over all closed curves in $K$ for every $p\in P$;
\vspace{1mm}
\item[\rm (c)] $f_p\theta$ has vanishing periods over all closed curves in $M$ for every $p\in Q$. 
\end{itemize}
Given a number $\epsilon>0$, there exists a homotopy $f^t \colon M\times P\to A_*$ 
$(t\in [0,1])$ such that $f^t_p := f^t(\cdotp,p)\colon M\to A_*$ is a nondegenerate holomorphic map
for every $(p,t)\in P\times [0,1]$ and the following conditions hold:
\begin{itemize}
\item[\rm (1)]  $f^t_p =f_p$ \ for every $(p,t)\in (P\times \{0\}) \cup (Q\times [0,1])$; 
\vspace{1mm}
\item[\rm (2)]  $\|f^t-f\|_{K\times P} <\epsilon$ for all $t \in [0,1]$;
\vspace{1mm}
\item[\rm (3)]  
$f^t_p\theta$ has vanishing periods on $K$ for every $(p,t)\in P\times [0,1]$;
\vspace{1mm}
\item[\rm (4)]  
$f^1_p\theta$ has vanishing periods on $M$ for every $p\in P$.
\end{itemize}
The same result holds if we replace the condition {\em vanishing periods} in parts
{\rm (b), (c), (3), (4)} by the condition  {\em vanishing real periods}. 
\end{theorem}

\begin{proof}
We shall use the tools developed in Sections \ref{sec:CI-lemma} and \ref{sec:PHP1}.
The proof is similar to that of Theorem \ref{th:PHP1}; we indicate the essential points and leave out the details.

We may assume that $A$ is not contained in any proper complex subspace of $\C^n$.
By \cite[Lemma 3.1]{AlarconForstneric2014IM}, its convex hull $\Co(A)\subset\C^n$ then 
equals $\C^n$. This gives the analogue of Lemma \ref{lem:CI} for loops in $A_*$. 
Choose an exhaustion of $M$
\[
	K=K_1\subset K_2\subset\cdots\subset \cup_{j=1}^\infty K_j=M
\]
by compact regular sublevel sets $K_j=\{x\in M: \rho(x)\le c_j\}$
of a strongly subharmonic exhaustion function $\rho\colon M\to\R$, with $K_1=\bar D$. 
We recursively build a sequence of homotopies 
\[
	F^t_{p,j} \colon K_j=\bar D_j\lra \C^n,\quad j=1,2,\ldots,
\]
consisting of nondegenerate $A$-immersions,
such that all relevant conditions hold (cf.\ conditions ($\alpha$)--($\gamma$) in the proof
of Theorem \ref{th:PHP1}).  The homotopy $F^t$, satisfying the conclusion of Theorem \ref{th:PHP2}, 
is obtained as the locally uniform limit 
\[
	F^t_p=\lim_{j\to\infty} F^t_{p,j},\quad (p,t)\in P\times [0,1]. 
\]

The noncritical case in the induction step is done exactly as in the proof of Theorem \ref{th:PHP1}.

In the critical case, we begin by extending the continuous family of holomorphic maps
\[
	f^t_{p,j} := dF^t_{p,j}/\theta\colon K_j\lra \Agot_*, \quad (p,t)\in P\times [0,1]
\]
from the set $K_j$ to an embedded arc $E$, attached with endpoints to $K_j$, such that, if $E$
closes inside $K_j$ to a loop $C$, then the periods of the extended maps vanish for $t=1$: 
\[	
	\int_C  f^1_{p,j} \,\theta=0, \quad p\in P.
\]
At the same time, we keep the already given values of these maps for the parameter values $p\in Q$. 
This is accomplished as in the proof of Theorem \ref{th:PHP1} by applying the analogue of Lemma \ref{lem:CI}
for loops in $A_*$. The rest of the proof goes through just as  in Theorem \ref{th:PHP1},
using a period dominating spray with values in $A_*$, Mergelyan's approximation theorem
for maps to $A_*$, etc. (See the paper \cite{AlarconForstneric2014IM} for these techniques.) 
We leave out further details.
\end{proof}

Let us denote by $\Oscr_*(M,A_*)$ the subset of $\Oscr(M,A_*)$ consisting of all nondegenerate maps 
$M\to A_*$ (see Definition \ref{def:nondegenerate}). We have the following general position theorem, 
which we state in the form of parametric h-principle with approximation
for the inclusion $\Oscr_*(M,A_*) \hookrightarrow \Oscr(M,A_*)$.

%
%
\begin{theorem} \label{th:PHP3}
Assume that $M$ is an open Riemann surface,  $K$ is a compact set in $M$, 
$A\subset \C^n$ is a closed conical complex subvariety of $\C^n$ such that $A_*=A\setminus \{0\}$ is an Oka manifold, 
and $Q\subset P$ are compact sets in a Euclidean space. Assume that $f \colon M\times P\to A_*$ is a continuous map 
satisfying the following two conditions:
\begin{itemize}
\item[\rm (a)] $f_p=f(\cdotp,p)\colon M\to A_*$ is a holomorphic map for every $p\in P$;
\vspace{1mm}
\item[\rm (b)] $f_p\in \Oscr_*(M,A_*)$ is a nondegenerate holomorphic map for every $p\in Q$.
\end{itemize}
Given a number $\epsilon>0$, there exists a homotopy $f^t \colon M\times P\to A_*$ 
$(t\in [0,1])$ such that $f^t_p := f^t(\cdotp,p)\in \Oscr(M,A_*)$ 
for every $(p,t)\in P\times [0,1]$ and the following conditions hold:
\begin{itemize}
\item[\rm (1)]  $f^t_p =f_p$ \ for every $(p,t)\in (P\times \{0\}) \cup (Q\times [0,1])$; 
\vspace{1mm}
\item[\rm (2)]  
$f^1_p\in \Oscr_*(M,A_*)$  is nondegenerate for every $p\in P$;
\vspace{1mm}
\item[\rm (3)]  $\|f^t-f\|_{K\times P} <\epsilon$ for every $t \in [0,1]$.
\end{itemize}
If $A_*$ is not assumed to be an Oka manifold, then the same result still
holds if $M$ is a compact bordered Riemann surface.
\end{theorem}

\begin{proof}
We may assume that $K$ is $\Oscr(M)$-convex and has nonempty interior.
Consider first the case when there exist finitely many $\C$-complete holomorphic vector fields
$V_1,\ldots, V_m$ tangent to $A_*$ which span the tangent space of $A_*$ at every point.
(This holds if $A$ is the null quadric $\Agot$, cf.\ \cite[Example 4.4]{AlarconForstneric2014IM}.)
Consider a map $\Psi\colon M\times P \times \C^N\to A_*$ of the form
\begin{equation}
\label{eq:Psi}
	\Psi(x,p,\zeta)=\phi^1_{\zeta_1 h_1(x,p)} \circ \cdots  \circ \phi^N_{\zeta_{N} h_{N}(x,p)} (f(x)),
\end{equation}
where $N$ is a large integer, $(x,p)\in M\times P$, $\zeta\in \C^N$, every $\phi^j_\zeta$ 
is the flow of one of the vector fields $V_1,\ldots, V_m$ (which may appear with repetitions),
and the functions $h_j\in \Cscr(M\times P)$ are such that $h_{j}(\cdotp,p)\in \Oscr(M)$ 
for every $p\in P$ and $h_{j}(\cdotp,p)=0$ for every $p\in Q$. A suitable choice of the functions 
$h_j$ ensures (by a standard transversality argument) that for a generic choice of 
$\zeta\in\C^N$ the homotopy
\begin{equation}\label{eq:ft}
	f^t :=\Psi(\cdotp,\cdotp,t\zeta)\colon M\times P\to A_*,\quad  t\in [0,1]
\end{equation}
has the stated properties. The approximation in (3) is achieved by choosing the point
$\zeta$ close to $0$. The details are similar as in 
\cite[proof of Theorem 3.2 (a)]{AlarconForstneric2014IM}, except that the
present situation is simpler since we need not fulfil any period conditions.

In general, the tangent bundle of $A_*$ is spanned by finitely many
holomorphic vector fields (not necessarily complete) in view of Cartan's theorem A.
By using their local holomorphic flows, we can obtain a deformation family \eqref{eq:Psi}
with the desired properties over the compact subset $K\subset M$ and with the parameter
$\zeta$ in a small ball $B\subset \C^N$. This furnishes a homotopy  \eqref{eq:ft} 
defined on $K\times P$ such that every map $f^1_p\colon K\to A_*$ $(p\in P)$ 
is nondegenerate on $K$. Since $A_*$ is an Oka manifold, the parametric Oka property 
with approximation (see \cite[Theorem 5.4.4]{Forstneric2011})
shows that the homotopy $(f^t)_{t\in[0,1]}$ can be approximated uniformly on $K\times P\times [0,1]$ 
by a homotopy $F^t \colon M\times P\to A_*,\  t\in [0,1]$ such that every map 
$F^t(\cdotp,p)\colon M\to A_*$ is holomorphic and the two homotopies agree at $t=0$
and for the parameter values $q\in Q$. If the approximation 
is sufficiently close, then clearly the map $F^1_p\colon M\to A_*$ is nondegenerate
for every $p\in P$. 
\end{proof}

The following is an immediate corollary to Theorem \ref{th:PHP3}.

\begin{corollary} \label{cor:WHEforOstar}
Assume that $M$ is an open Riemann surface and $A\subset \C^n$ is a closed conical subvariety
of $\C^n$ such that $A_*=A\setminus \{0\}$ is an Oka manifold. Then the inclusion 
\begin{equation}\label{eq:Ostar}
	\Oscr_*(M,A_*) \longhookrightarrow \Oscr(M,A_*)
\end{equation}
is a weak homotopy equivalence.
\end{corollary}

Let us denote by 
\[
	\Igot_{A,*}(M,\C^n)
\]
the space of all nondegenerate $A$-immersions $M\to\C^n$ with the compact-open topology (see Definition \ref{def:nondegenerate}).
By \cite[Lemma 2.3]{AlarconForstnericLopez2015MZ}, a holomorphic null curve $M\to\C^n$ is nondegenerate 
if and only if it is nonflat. Thus, 
\[
	\Igot_{\Agot,*}(M,\C^n) = \Ngot_*(M,\C^n)
\]
is the space of all nonflat null holomorphic curves $M\to\C^n$.
(Note that, in \cite{AlarconForstnericLopez2015MZ}, the space of nonflat holomorphic null curves
is denoted by $\Ngot_{\mathrm{nf}}(M,\C^n)$, and likewise for the space of conformal minimal
immersions. The definition of nondegeneracy in \cite[Definition 2.2]{AlarconForstnericLopez2015MZ} 
is different from our Definition \ref{def:nondegenerate} and will not be used here.)

%
%
%
%
\begin{theorem}\label{th:WHE-A}
{\rm (The weak homotopy equivalence principle for $A$-immersions.)} 
Let $M$ be an open Riemann surface, and let $A\subset \C^n$ be  a closed
conical complex subvariety of $\C^n$ such that $A_*=A\setminus \{0\}$ is an Oka manifold. 
Fix a nowhere vanishing holomorphic $1$-form $\theta$ on $M$. Then the map
\[
	\Igot_{A,*}(M,\C^n) \longrightarrow \Oscr(M,A_*),\quad  F\longmapsto dF/\theta
\]
is a weak homotopy equivalence.
\end{theorem}

Since the inclusion $\Oscr(M,A_*) \hookrightarrow \Cscr(M,A_*)$ is a weak homotopy
equivalence when $A_*$ is an Oka manifold (cf.\ \cite[Corollary 5.4.8]{Forstneric2011}), 
Theorem \ref{th:WHE-A} is equivalent to saying  that the map 
\[
	\Igot_{A,*}(M,\C^n) \lra \Cscr(M,A_*), \quad F\longmapsto dF/\theta
\]
is a weak homotopy equivalence.  Note that Theorem \ref{th:WHE-A} generalizes 
\cite[Theorem 2.6]{AlarconForstneric2014IM}, the latter result 
providing the basic h-principle in this context. 

\begin{proof}[Proof of Theorem \ref{th:WHE-A}] 
Note that the map $F\to dF/\theta$ takes $\Igot_{A,*}(M,\C^n)$ into $\Oscr_*(M,A_*)$.
In view of Corollary \ref{cor:WHEforOstar}, we only need to show that it induces 
a weak homotopy equivalence $\Igot_{A,*}(M,\C^n) \to \Oscr_*(M,A_*)$, 
i.e., the induced map 
\begin{equation}\label{eq:pik}
	\pi_k(\Igot_{A,*}(M,\C^n))  \longrightarrow \pi_k(\Oscr_*(M,A_*))
\end{equation}
is bijective for every $k\in\Z_+$. 
We proceed as in the proof of Theorem \ref{th:WHE1}.
We begin by proving that the map \eqref{eq:pik} is surjective.
Let $P=\S^k$ be the real $k$-sphere and $Q=\emptyset$. A map $f\colon P\to \Oscr_*(M,A_*)$
is naturally identified with a map $f\colon M\times P \to A_*$ such that 
$f_p=f(\cdotp,p)\in \Oscr_*(M,A_*)$ for every $p\in P$. 
Theorem \ref{th:PHP2} shows that we can deform the family $(f_p)_{p\in P}$ in $\Oscr_*(M,A_*)$
to a  family  $(f^1_p)_{p\in P}$ in $\Oscr_*(M,A_*)$ 
such that the $1$-form $f^1_p\theta$ has vanishing periods for every $p\in P$. 
By choosing a base point $x_0\in M$ and setting 
\begin{equation}\label{eq:integrals}
	F_p(x)=\int_{x_0}^x  f^1_p\theta,\qquad x\in M,\ \ p\in P
\end{equation}
we obtain a continuous family of nondegenerate $A$-immersions $F_p\colon M\to \C^n$ $(p\in P)$. 
This proves that the map \eqref{eq:pik} is surjective.

It remains to shows that the map \eqref{eq:pik} is also injective.
Let $P=\overline\B^{k+1}$ be the closed ball in $\R^{k+1}$ and $Q=\S^k=b\B^{k+1}$.
Consider a continuous family of nondegenerate holomorphic maps $f_p\colon M\to A_*$ $(p\in P)$ 
such that, for every $p\in Q$, the $1$-form $f_p \theta$ has vanishing periods over all closed curves in $M$.
(The latter property implies that the integral $\int^x  f_p\theta$ $(x\in M)$ 
is a nondegenerate $A$-immersion for every $p\in Q$.)
By Theorem \ref{th:PHP2} we can deform the family $(f_p)_{p\in P}$ in $\Oscr_*(M,A_*)$
to a family $(f^1_p)_{p\in P}$ in $\Oscr_*(M,A_*)$ by a deformation that is fixed for $p\in Q$ 
and such that $f^1_p\theta$ has vanishing periods on $M$ for all $p\in P$. Their integrals
given by \eqref{eq:integrals}  then form a continuous family of nondegenerate $A$-immersions.
\end{proof}

\begin{proof}[Proof of Theorem  \ref{th:WHE2}]
Since the null quadric $\Agot_*$ \eqref{eq:Agot} is an Oka manifold, 
Theorem \ref{th:WHE-A} includes the second claim in Theorem \ref{th:WHE2} as a special case. 
The first claim concerning the map $\Mgot_*(M,\R^n)\to \Oscr(M,\Agot_*)$
then follows from the observation that the maps $\phi$, $\psi$ in the diagram \eqref{eq:diagram} are 
weak homotopy equivalences at the same time.
\end{proof}


\section{Strong parametric h-principles for sources of finite topological type}
\label{sec:strong-h-principle}

Let $M$ be an open Riemann surface and $n\geq 3$.  We might as well assume that $M$ is connected.  
By now we know that the maps in the commuting diagram
\begin{equation} \label{eq:diagram2}
\xymatrix{
	\Ngot_*(M,\C^n)  \ar[r]^\phi \ar[d]_\Re  &  \Oscr_*(M,\Agot_*)  \ar@{^{(}->}[r]^{j_1}  & \Oscr(M,\Agot_*) \ar@{^{(}->}[r]^{j_2} & \Cscr(M,\Agot_*)\\ 
	\Re\Ngot_*(M,\C^n)   \ar@{^{(}->}[r]^\iota    &  \Mgot_*(M,\R^n)  \ar[u]_\psi &  &
}
\end{equation}
are weak homotopy equivalences.  Recall that all these spaces carry the compact-open topology.
As noted earlier, the left vertical map is a homotopy equivalence by continuity of the Hilbert transform that 
takes $u\in\Re\Ngot_*(M,\C^n)$ to its harmonic conjugate $v$ with $v(p)=0$, where $p\in M$ is any chosen base point.  

In this final section, we assume that $M$ has finite topological type.  This means that the fundamental group of $M$ is finitely generated or, in other words, that $M$ has the homotopy type of a \textit{finite} wedge of circles.  Equivalently, by a theorem of Stout \cite[Theorem 8.1]{Stout1965}, $M$ can be obtained from a compact Riemann surface by removing a finite number of mutually disjoint points and closed discs.  Thus, equivalently, $M$ has a strictly subharmonic Morse exhaustion with finitely many critical points.

By \cite[Theorem 9]{Larusson2015}, when $M$ has finite topological type, $\Oscr(M,\Agot_*)$ is an absolute neighborhood retract (in the category of metrizable spaces).  This relies on $\Agot_*$ being Oka.  Also, $\Cscr(M,\Agot_*)$ is an ANR.  This follows from work of Milnor and of Smrekar and Yamashita \cite{SmrekarYamashita2009} and was formulated as \cite[Proposition 7]{Larusson2015}.  As an open subspace of an ANR, $\Oscr_*(M,\Agot_*)$ is an ANR.  

An ANR has the homotopy type of a CW complex. By a theorem of Whitehead  \cite[p.~74]{May1999},
a weak homotopy equivalence between CW complexes is a homotopy equivalence.
Hence, as weak homotopy equivalences between ANRs, the inclusions $j_1$ and $j_2$ are homotopy equivalences.
In fact, $j_2$ is the inclusion of a strong deformation retract \cite[Theorem 1]{Larusson2015}.  Namely, the inclusion of a closed subspace is a cofibration (in the sense of Hurewicz) when both spaces are ANRs.  Being a cofibration and a homotopy equivalence, $j_2$ is the inclusion of a strong deformation retract.

We conclude the paper by proving that when $M$ has finite topological type, the maps $\iota$, $\phi$ and $\psi$ (the latter suitably restricted) are inclusions of strong deformation retracts.  In particular they are homotopy equivalences.  In addition to the theory of ANRs, recent results in \cite{AlarconForstneric2014IM} and \cite{AlarconForstnericLopez2015MZ} on certain spaces of maps being Banach manifolds and therefore locally contractible play a key role.
For spaces of holomorphic maps from strongly pseudoconvex Stein domains to complex manifolds, 
such results first appeared in \cite{Forstneric2007}.  
We need to open up the proofs of these results and develop them further (see Lemma \ref{lem:contractible}).

%
%
\begin{theorem} \label{th:anr}
Let $M$ be a connected open Riemann surface of finite topological type.  Let $n\geq 3$.  
Then the metrizable spaces $\Mgot_*(M,\R^n)$, $\Ngot_*(M,\C^n)$ and $\Re\Ngot_*(M,\C^n)$ are 
absolute neighborhood retracts.
\end{theorem}

Before proving the theorem, we state and prove the following corollary.

\begin{corollary}  \label{cor:strong}
Let $M$ be a connected open Riemann surface of finite topological type.  Let $n\geq 3$.  The six maps in the diagram \eqref{eq:diagram2} are homotopy equivalences.  Moreover, the inclusion $\iota$ and the injections
\[ \psi: \{u\in\Mgot_*(M,\R^n):u(p)=0\}\to \Oscr_*(M,\Agot_*), \]
\[ \phi: \{F\in\Ngot_*(M,\C^n):F(p)=0\}\to \Oscr_*(M,\Agot_*), \]
where $p\in M$ is any chosen base point, are inclusions of strong deformation retracts. 
\end{corollary}

\begin{proof}
By the above, all that remains to observe is that $\iota$, $\psi$ and $\phi$ have closed images, and that the subspaces $\{u\in\Mgot_*(M,\R^n):u(p)=0\}$ and $\{F\in\Ngot_*(M,\C^n):F(p)=0\}$ of $\Mgot_*(M,\R^n)$ and $\Ngot_*(M,\C^n)$, respectively, are ANRs.  The obvious homeomorphism $\Ngot_*(M,\C^n)\to \{F\in\Ngot_*(M,\C^n):F(p)=0\}\times\C^n$, $F\mapsto (F-F(0), F(0))$, with inverse $(G,a)\mapsto G+a$, shows that $\{F\in\Ngot_*(M,\C^n):F(p)=0\}$ is an ANR since $\Ngot_*(M,\C^n)$ is (and conversely).  Similarly, $\{u\in\Mgot_*(M,\R^n):u(p)=0\}$ is an ANR.
\end{proof}

\begin{remark}  \label{rem:general}
Analogous results hold for the maps 
\begin{equation*}
\xymatrix{
	\Igot_{A,*}(M,\C^n)  \ar[r]^\phi &  \Oscr_*(M,A_*)  \ar@{^{(}->}[r] & \Oscr(M,A_*) \ar@{^{(}->}[r] & \Cscr(M,A_*)
}
\end{equation*}
whenever $A$ is a closed irreducible conical subvariety of $\C^n$ such that $A_*=A\setminus \{0\}$ is an Oka manifold. 
%
%
This holds in particular when $A=\C^n$ for any $n\in\N$ 
and gives a homotopy equivalence $\Igot(M,\C^n) \longrightarrow \Cscr(M,\S^{2n-1})$
in Theorem \ref{th:immersions}.
%
%
\end{remark}

Before proceeding, we introduce some notation. Let $A\subset \C^n$ be  a closed conical complex subvariety of $\C^n$ such that $A_*=A\setminus \{0\}$ is smooth.  Let $M$ be an open Riemann surface.  We shall write $\Oscr_{*,0}(M,A_*)$ for the space of nondegenerate holomorphic maps  $M\to\Agot_*$ (in the sense of Definition \ref{def:nondegenerate}) with vanishing periods. If $M$ is a compact bordered Riemann surface, we shall use the notation
\[
	\Ascr_{*,0}(M,A_*)  \subset \Ascr_{*}(M,A_*) \subset  \Ascr(M,A_*),
\]
where $\Ascr(M,A_*)$ is the space of continuous maps $M\to A_*$ that are holomorphic in the interior of $M$,
$\Ascr_{*}(M,A_*)$ is the subspace of nondegenerate maps, and $\Ascr_{*,0}(M,A_*)$ 
is the subspace of nondegenerate maps with vanishing periods.  

When $M$ is a compact $\Cscr^1$-bordered Riemann surface, $\Ascr_{*,0}(M,A_*)$ is a complex Banach manifold;
more precisely, it is a locally closed complex Banach submanifold of finite codimension in the complex Banach manifold 
$\Ascr(M,A_*)$ \cite[Theorem 2.3(b)]{AlarconForstneric2014IM}.  
The analogous theorem for vanishing real periods for 
$A=\Agot$ is \cite[Theorem 3.1(b)]{AlarconForstnericLopez2015MZ}.

To prove Theorem \ref{th:anr}, we need the following lemma.  

\begin{lemma}  \label{lem:contractible}
Assume that $A$ is a closed irreducible conical subvariety of $\C^n$ such that $A_*=A\setminus \{0\}$ is smooth.
Let $M$ be an open Riemann surface and let  $\rho:M\to[0,\infty)$ be a smooth exhaustion. 
Let $L_0\supset L_1\supset \cdots \supset K$ be compact smoothly bordered subsurfaces of $M$ of the form 
$\rho^{-1}([0,c])$, such that $K$ contains all the critical points of $\rho$.  Let $f\in\Ascr_{*,0}(M,A_*)$ and let 
$W$ be a neighborhood of $f|_K$ in $\Ascr_{*,0}(K,A_*)$.  Then there are contractible neighborhoods $C_m$ of 
$f|_{L_m}$ in $\Ascr_{*,0}(L_m,A_*)$ such that  $C_m|_{L_{m+1}}\subset C_{m+1}$ and $C_m|_K\subset W$ 
for all $m\geq 0$. 
\end{lemma}

\begin{proof}
We must examine the proof of \cite[Theorem 2.3(b)]{AlarconForstneric2014IM}
concerning the Banach manifold structure of $\Ascr_{*,0}(M,A_*)$ when $M$ is a 
compact bordered Riemann surface.

Let $\theta$ be a nowhere vanishing  holomorphic $1$-form on $M$.
Since $K$ contains all the critical points of $\rho$, there is a homology basis
$\Bcal=\{\gamma_i\}_{i=1,\ldots,l}$ of $H_1(M;\Z)$ whose support $|\Bcal| = \cup_{j=1}^l |\gamma_j|$ 
is contained in $K$ and is Runge in $M$. Let $\Pcal\colon  \Oscr(M,A_*) \to (\C^n)^l$ denote 
the period map associated to $\Bcal$ and $\theta$:
\[
	\Pcal(f) =\left(\int_{\gamma_j} f\theta \right)_{j=1,\ldots,l} \in (\C^n)^l,\quad f\in \Oscr(M,A_*).
\]

Fix a map $f\in \Oscr_{*,0}(M,A_*)$. 
Let $M_0$ be a compact smoothly bounded domain in $M$ (a sublevel set of $\rho$) 
with the same topology as $M$ and containing $L_0$.
By \cite[Lemma 5.1]{AlarconForstneric2014IM}, there exist an
integer $N\in\N$ and a dominating and period dominating holomorphic spray map 
$\Psi\colon M_0 \times \B^N \to A_*$ of class $\Ascr(M_0)$ with the core $\Psi(\cdotp,0)=f$. 
(More precisely, $\Psi$ is continuous and is holomorphic on $\mathring M_0 \times \B^N$.)
Consider the partial differential
\[
	\Theta_x = \di_\zeta \Psi(x,\zeta)\big|_{\zeta=0} \colon \C^N\lra T_{f(x)}A_*, \quad x\in M_0.
\]
The domination property of $\Psi$ means that $\Theta_x$ is surjective for each $x\in M_0$.
Hence, we have $M_0 \times \C^N= \ker \Theta \oplus E$  where $E$ is a complex vector subbundle 
of class $\Ascr(M_0)$ and rank $k=\dim_\C A_*$. 
(See \cite[proof of Theorem 1.2]{Forstneric2007}, where the reader can also find the relevant references.) 
Note that $E$, being a complex vector bundle over
a nonclosed Riemann surface, is trivial, $E\cong M_0\times \C^k$ as complex vector bundles
of class $\Ascr(M_0)$, so we simply identify $E$ with $M_0\times \C^k$ in the sequel.
We now restrict $\Psi$ to $E=M_0\times \C^k$ and denote it by the same letter.

By the implicit function theorem, there is $\epsilon_0\in (0,1)$ such that
$\Psi(x,\cdotp)\colon \epsilon_0\B^k \to A_*$ is an injective holomorphic map
for every $x\in M_0$, hence biholomorphic onto its image. It follows that the map
\begin{equation}\label{eq:Phi}
	\Phi\colon \Ascr(M_0,\epsilon_0\B^k) \to \Ascr(M_0,A_*),
	\quad \Phi(g)(x) := \Psi(x,g(x))\ \ (x\in M_0)
\end{equation}
takes $\Ascr(M_0,\epsilon_0\B^k)$ bijectively onto an open neighborhood of 
$f|_{M_0}=\Phi(0)$ in $\Ascr(M_0,A_*)$. Since $f$ is nondegenerate, we may decrease $\epsilon_0>0$ 
to ensure that the range of $\Phi$ is contained in the open 
subset $\Ascr_*(M_0,A_*)$ of $\Ascr(M_0,A_*)$ consisting of nondegenerate maps.

Consider now the subset 
\[
	\{g\in \Ascr(M_0,\epsilon_0\B^k) : \Pcal(\Phi(g))=0\}
\]  
of $\Ascr(M_0,\epsilon_0\B^k)$ which corresponds via $\Phi$ to maps $M_0\to A_*$ with vanishing periods.
By the period domination property of the spray $\Psi$ (and hence of the map $\Phi$ \eqref{eq:Phi}), 
the differential of the period map
\[	
	\Ascr(M_0,\epsilon_0\B^k) \ni g\longmapsto \Pcal(\Phi(g)) \in (\C^n)^l
\] 
at $g=0$ is surjective. Let us denote this differential by 
\begin{equation}\label{eq:D}
	D= d_0(\Pcal\circ \Phi) : \Ascr(M_0,\C^k)\lra (\C^n)^l.
\end{equation}
Pick $h_1,\ldots,h_{nl} \in \Ascr(M_0,\epsilon_0\B^k)$ such that the vectors
$D(h_1),\ldots,D(h_{nl})\in (\C^n)^l$ span $(\C^n)^l$. Note that the map $\Phi$ given by \eqref{eq:Phi}
also applies to continuous maps $g\colon L\to \epsilon_0\B^k$ defined on a subset $L\subset M_0$, and the 
period map $\Pcal(\Phi(g))$ is defined whenever the domain $L$ of $g$ contains the support
$|\Bcal|$ of the homology basis. Hence, the map $D$ \eqref{eq:D} is well defined on 
$\Cscr(L,\C^k)$ whenever $|\Bcal| \subset L \subset M_0$. Taking $L=|\Bcal|$, it follows that the Banach space  
$\Cscr(|\Bcal|,\C^k)$ decomposes as a direct sum of closed complex Banach subspaces
\[
	\Cscr(|\Bcal|,\C^k) = \left( \ker D|_{\Cscr(|\Bcal|,\C^k)}\right) \oplus 
	\span_\C \{h_1|_{|\Bcal|},\ldots,h_{nl}|_{|\Bcal|}\}
	= \Lambda \oplus H. 
\] 
By the implicit function theorem in Banach spaces, there is a number $\epsilon_1\in (0,\epsilon_0)$
such that for every $g\in \Lambda=\ker D|_{\Cscr(|\Bcal|,\C^k)}$ with $\Vert g\Vert_{0,|\Bcal|} < \epsilon_1$, 
there exists an element of the form 
\begin{equation}\label{eq:tildeg}
	\tilde g = g + \sum_{j=1}^{nl} c_j(g) h_j \in \Cscr(|\Bcal|,\epsilon_0\B^k) 
\end{equation}
satisfying the period vanishing equation
\begin{equation}\label{eq:periodzero}
	\Pcal(\Phi(\tilde g))=0.  
\end{equation}
Here, $c_j$ are smooth bounded complex functions on the set  
$\{g\in \Lambda : \Vert g\Vert_{0,|\Bcal|} < \epsilon_1\}$ furnished by the implicit function theorem.
Note that \eqref{eq:tildeg} is a local representation of the solution set of 
\eqref{eq:periodzero} as a graph over the linear subspace $\Lambda$ near the origin.

Recall that $h_1,\ldots, h_{nl}  \in \Ascr(M_0,\C^k)$. Hence, if $L$ is any smoothly bounded compact set
with $|\Bcal|\subset L \subset M_0$ and $g\in \Ascr(L,\epsilon_1\B^k)$, then the formula  \eqref{eq:tildeg} yields a map 
\[
	\tilde g = \psi_L(g) :=  g + \sum_{j=1}^{nl} c_j(g|_{|\Bcal|}) h_j|_{L}  \in \Ascr(L,\C^k)
\]
satisfying the period vanishing condition \eqref{eq:periodzero} provided that $\|\tilde g\|_{0,L}< \epsilon_0$
(so that $\Phi(\tilde g)$ is defined). 
Note that for any pair of compacts $L,L'$ with $|\Bcal| \subset L\subset L'\subset M_0$, we have
\begin{equation}\label{eq:restriction}
	\psi_{L}(g|_L) = \psi_{L'}(g)\big|_L\quad \text{for every}\ \ g\in \Ascr(L',\epsilon_1\B^k).
\end{equation}

Since the functions $c_j$ are bounded on a neighborhood of the origin in $\Lambda$, 
there is a number $\epsilon$ with $0<\epsilon <\epsilon_1$ such that  
\[
	V_0 := \bigl\{\psi_{M_0}(g) : g \in \Ascr(M_0,\epsilon\B^k),\ D(g)=0 \bigr\}
	\subset \Ascr(M_0,\epsilon_0\B^k).
\]
Note that $V_0$ is a neighborhood of the origin in the space of solutions of 
\eqref{eq:periodzero} over $M_0$. Furthermore, being a smooth graph over the convex set  
$\{g\in \Ascr(M_0,\epsilon\B^k) : D(g)=0\}$, $V_0$ is contractible. Similarly, for every $m\geq 0$, the set
\[
	W_m := \bigl\{\psi_{L_m}(g)  : g \in \Ascr(L_m,\epsilon\B^k),\ D(g)=0 \bigr\}   \subset \Ascr(L_m,\epsilon_0\B^k)
\]
is a contractible neighborhood of the origin in the space of solutions of \eqref{eq:periodzero} over $L_m$. 
Taking into account that for any $g\in \Ascr(L_m,\C^k)$ we have
$\|g\|_{L_{m+1}} \le \|g\|_{L_{m}}$ by the maximum principle, 
the formula \eqref{eq:restriction} shows that the restriction map associated to 
the inclusion $L_m\supset L_{m+1}$ maps $W_m$ into $W_{m+1}$ for every $m\geq 0$.
It follows that
\[
	C_m :=\{\Phi(\tilde g) : \tilde g \in W_m\}  \subset \Ascr_{*,0}(L_m,A_*)
\]
is a contractible neighborhood of $f|_{L_m}$ in $\Ascr_{*,0}(L_m,A_*)$, and the restriction map associated to 
$L_m\supset L_{m+1}$ takes $C_m$ into $C_{m+1}$ for every $m\geq 0$.
By choosing $\epsilon>0$ small enough, we can also ensure that the restriction map associated
to $L_m\supset K$ maps $C_m$ into a given neighborhood of $f|_{K}$ in $\Ascr_{*,0}(K,A_*)$.
\end{proof}

The analogous lemma for vanishing real periods is proved similarly.

\begin{proof}[Proof of Theorem \ref{th:anr}]
As noted in the proof of Corollary \ref{cor:strong}, $\Ngot_*(M,\C^n)$ is an ANR if and only if $\{F\in\Ngot_*(M,\C^n):F(p)=0\}$ is.  We identify the latter space with $X=\Oscr_{*,0}(M,\Agot_*)$ and prove that $X$ is an ANR.

Let $\mathscr U$ be an open cover of $X$.  We need to produce a refinement $\mathscr V$ of $\mathscr U$ such that if $P$ is a simplicial complex with a subcomplex $Q$ containing all the vertices of $P$, then every continuous map $\phi_0:Q\to X$ such that for each simplex $\sigma$ of $P$, $\phi_0(\sigma\cap Q)\subset V$ for some $V\in\mathscr V$, extends to a continuous map $\phi:P\to X$ such that for each simplex $\sigma$ of $P$, $\phi(\sigma)\subset U$ for some $U\in\mathscr U$.  This is the Dugundji-Lefschetz property, which is equivalent to $X$ being an ANR (see \cite[Theorem 5.2.1]{VanMill-book}).

Let $\rho:M\to[0,\infty)$ be a strictly subharmonic Morse exhaustion with finitely many critical points.  Let $f\in X$ and choose $U\in\mathscr U$ with $f\in U$.  There are compact $\Cscr^1$-bordered subsurfaces $L_0\supset K_0 \supset L_1 \supset K_1 \supset \cdots \supset K$ of $M$ of the form $\rho^{-1}([0,c])$, such that for every $m\geq 0$, $K_m\subset \mathring  L_m$, $K_m$ is $\mathscr O(M)$-convex, and $K$ contains all the critical points of $\rho$, so $M$ deformation-retracts onto $L_m$.  By choosing $K$ large enough, we may assume that $f|_K$ has a neighborhood $W$ in $\Ascr_{*,0}(K,\Agot_*)$ such that if $g\in X$ and $g|_K\in W$, then $g\in U$.

By Lemma \ref{lem:contractible}, there is a contractible neighborhood $C_m$ of $f|_{L_m}$ in $\Ascr_{*,0}(L_m,\Agot_*)$ such that $C_m|_{L_{m+1}}\subset C_{m+1}$ and $C_m|_K\subset W$ for all $m\geq 0$.  Let $V_m$ be the neighborhood of $f$ in $X$ consisting of all $g\in X$ with $g|_{L_m}\in C_m$.  Then $V_0\subset V_1\subset\cdots\subset U$.  Let the refinement $\mathscr V$ of $\mathscr U$ consist of all such open sets $V_0$, one for each $f\in X$.

Let $P$, $Q$, and $\phi_0$ be as above, so for each simplex $\sigma$ of $P$, $\phi_0(\sigma\cap Q)\subset V_0^\sigma$ for some $V_0^\sigma\in\mathscr V$ associated to some $f\in X$ as above.  Adorn the other associated sets $K_m$, $L_m$, $C_m$, and $V_m$ with a superscript $\sigma$ as well.  Let $P_m=P^m\cup Q$, where $P^m$ is the $m$-skeleton of $P$.  We shall construct continuous maps $\phi_m:P_m\to X$, $m\geq 1$, such that $\phi_m|_{P_{m-1}}=\phi_{m-1}$ and for every simplex $\sigma$ of $P$, $\phi_m(\sigma\cap P_m)\subset V_m^\sigma$.  The map $\phi:P\to X$ with $\phi|_{P_m}=\phi_m$ for each $m\geq 0$ will then be continuous with respect to the Whitehead topology on $P$, and for each simplex $\sigma$ of $P$, $\phi(\sigma)\subset U$ for some $U\in\mathscr U$.

Suppose $\phi_m$, $m\geq 0$, is given.  Let $\sigma$ be an $(m+1)$-simplex of $P$ but not of $Q$.  The interior $\sigma\setminus\partial\sigma$ of $\sigma$ does not intersect $P_m$.  Also, the interiors of distinct $(m+1)$-simplices do not intersect.  We need to suitably extend $\phi_m|_{\partial\sigma}$ to $\sigma$.  We have $\phi_m(\partial\sigma)\subset \phi_m(\sigma\cap P_m)\subset V_m^\sigma$.  Since $C_m^\sigma$ is contractible, the composition of $\phi_m:\partial\sigma\to V_m^\sigma$ and the restriction map $V_m^\sigma \to C_m^\sigma$ extends by the cofibration $\partial\sigma\hookrightarrow\sigma$ to a continuous map $\alpha:\sigma\to C_m^\sigma$.

We may view $\alpha$ as a continuous map $M\times\partial\sigma \cup L_m^\sigma\times\sigma \to \Agot_*$.  Now $L_m^\sigma$ is a deformation retract of $M$, so $M\times\partial\sigma \cup L_m^\sigma\times\sigma$ is a deformation retract of $M\times\sigma$.  Thus $\alpha$ extends to a continuous map $\alpha:M\times\sigma\to \Agot_*$.  Note that $\alpha(\cdot,t)\in X$ for all $t\in\partial\sigma$, and $\alpha(\cdot,t)|_{L_m^\sigma}\in\Ascr_{*,0}(L_m^\sigma,\Agot_*)$ for all $t\in\sigma$.  

By combining the parametric h-principles with approximation in Theorems \ref{th:PHP2} and \ref{th:PHP3} and the parametric Oka property with approximation of the Oka manifold $\Agot_*$, $\alpha$ may be deformed to a continuous map $\beta:X\times\sigma\to \Agot_*$ such that $\beta(\cdot,t)=\alpha(\cdot,t)$ for all $t\in\partial\sigma$, $\beta(\cdot,t)\in X$ for all $t\in\sigma$, and $\beta$ uniformly approximates $\alpha$ as closely as desired on $K_m^\sigma\times\sigma$.  If the approximation is close enough, then defining $\phi_{m+1}(t)=\beta(\cdot,t)$ for $t\in\sigma$ gives a continuous extension of $\phi_m|_{\partial\sigma}$ to $\sigma$ with $\phi_{m+1}(\sigma)\subset V_{m+1}^\sigma$.

The proof that $\Mgot_*(M,\R^n)$ is an ANR is analogous.  Finally, $\Re\Ngot_*(M,\C^n)$ is an ANR because by continuity of the Hilbert transform, the map $\Ngot_*(M,\C^n)\to\Re\Ngot_*(M,\C^n)\times\R^n$, $F\mapsto(\Re F, \Im F(p))$, where $p\in M$ is any chosen base point, is a homeomorphism, so $\Re\Ngot_*(M,\C^n)$ is a retract of $\Ngot_*(M,\C^n)$.
\end{proof}

Once again, let $M$ be an open Riemann surface and $n\geq 3$.  Returning to the first paragraph of the introduction, we have seen that the spaces $\Mgot_*(M,\R^n)$ and $\Ngot_*(M,\C^n)$, as well as the other four spaces in the diagram \eqref{eq:diagram2}, have the same \lq\lq rough shape\rq\rq.  More precisely, they all have the same weak homotopy type, and when $M$ has finite topological type even the same strong homotopy type, as the space $\mathfrak H$ of continuous maps from a wedge of circles to $\Agot_*^{n-1}$.  The real part projection $\C^n\to\R^n$ gives $\Agot_*^{n-1}$ the structure of a fibre bundle, whose fibre is the $(n-2)$-sphere, 
over $\R^n\setminus\{0\}$, which is homotopy equivalent to the $(n-1)$-sphere.  
The structure of $\mathfrak H$ can therefore be understood in terms of spheres and their loop spaces.  In particular, the homotopy groups of $\mathfrak H$ can be calculated in terms of homotopy groups of spheres.  We leave this for another day.


\subsection*{Acknowledgements}
F.\ Forstneri\v c is supported in part  by research program P1-0291 and grants J1-5432 and J1-7256 from ARRS, Republic of Slovenia.  F.~L\'arusson is supported in part by Australian Research Council grant DP150103442.  

A part of the work on this paper was done while F.\ Forstneri\v c was visiting the School of Mathematical Sciences at the University of Adelaide in January and February 2016.  He would like to thank the University of Adelaide for hospitality and the Australian Research Council for financial support.

The authors wish to  thank Antonio Alarc\'on for helpful discussions.  


{\bibliographystyle{abbrv}
\bibliography{bibFL}}

\vskip 5mm

\noindent Franc Forstneri\v c

\noindent Faculty of Mathematics and Physics, University of Ljubljana, Jadranska 19, SI--1000 Ljubljana, Slovenia

\noindent Institute of Mathematics, Physics and Mechanics, Jadranska 19, SI--1000 Ljubljana, Slovenia

\noindent e-mail: {\tt franc.forstneric@fmf.uni-lj.si}

\vskip 0.5cm

\noindent Finnur L\'arusson

\noindent School of Mathematical Sciences, University of Adelaide, Adelaide SA 5005, Australia

\noindent e-mail:  {\tt finnur.larusson@adelaide.edu.au}

\end{document}